\documentclass{amsart}
\usepackage[utf8]{inputenc}

\date{\today}

\title{Wetzel Families and the continuum}
\author{Jonathan Schilhan and Thilo Weinert}
\address{School of Mathematics, University of Leeds, Leeds, LS2 9JT, UK}
\email{j.schilhan@leeds.ac.uk}
\address{Università degli Studi di Udine, Dipartimento di Scienze Matematiche, Informatiche e Fisiche (DMIF), via delle Scienze 206, 33100 Udine, Italy}
\email{thilo.weinert@uniud.it}
\thanks{The first author was supported by a UKRI Future Leaders Fellowship [MR/T021705/2]. The second author was supported by FWF Lise-Meitner grant no. M3037. No data are associated with this article. We would like to thank Martin Goldstern for suggesting the name ``Wetzel family" in honor of John Edwin Wetzel. The second author would also like to thank Prof. Dr. Jürgen Elstrodt for helpful communication via email. Both authors would like to thank the anonymous referee for their work.}
\usepackage{amsmath}
\usepackage{amssymb}
\usepackage{bbm}
\usepackage{tikz-cd}
\usepackage{enumitem}
\usepackage{url}
\usepackage{hyperref}
\usepackage{todonotes}

\DeclareMathOperator{\cf}{cf}

\DeclareMathOperator{\dom}{dom}

\DeclareMathOperator{\proj}{proj}

\DeclareMathOperator{\RE}{Re}
\newcommand{\axiom}[1]{\mathsf{#1}}
\newcommand{\ZFC}{\axiom{ZFC}}

\newcommand{\CH}{\axiom{CH}}

\newcommand{\ZF}{\axiom{ZF}}
\newcommand{\MA}{\axiom{MA}}

\newcommand{\GCH}{\axiom{GCH}}
\newcommand{\PFA}{\axiom{PFA}}

\DeclareMathOperator{\Hol}{\mathcal{H}(\mathbb{C})}

\newcommand{\QH}{{\mathbb{Q}(H)}}

\newcommand{\lth}[1]{\vert #1 \vert}

\theoremstyle{plain}
\newtheorem{thm}{Theorem}[section]
\newtheorem{lemma}[thm]{Lemma}
\newtheorem{prop}[thm]{Proposition}
\newtheorem{cor}[thm]{Corollary}
\newtheorem{claim}[thm]{Claim}

\theoremstyle{definition}
\newtheorem{definition}[thm]{Definition}

\newtheorem{quest}[thm]{Question}
 
\begin{document}

\begin{abstract}
We provide answers to a question brought up by Erd\H{o}s about the construction of Wetzel families in the absence of the continuum hypothesis --- a \emph{Wetzel family} is a family $\mathcal{F}$ of entire functions on the complex plane which pointwise assumes fewer than $\vert \mathcal{F} \vert$ values. To be more precise, we show that the existence of a Wetzel family is consistent with all possible values $\kappa$ of the continuum and, if $\kappa$ is regular, also with Martin's Axiom. In the particular case of $\kappa = \aleph_2$ this answers the main open question asked by Kumar and Shelah in \cite{Kumar2017}. In the buildup to this result, we are also solving an open question of Zapletal on strongly almost disjoint functions from \cite{Zapletal1997}. We also study a strongly related notion of sets exhibiting a universality property via mappings by entire functions 
and show that these consistently exist while the continuum equals $\aleph_2$.
\end{abstract}

\maketitle

\section{Introduction}

This paper is an investigation related to Wetzel's problem which comes from analysis yet has surprising set-theoretical aspects. While the subjects of analysis and set theory might nowadays be conceived as somewhat distant to each other, there are some examples of topics belonging to them both. Among the most prominent one is the problem of sets of uniqueness, cf. \cite{KechrisLouveau1987}, which led Cantor to investigate sets of real numbers and subsequently to the founding of set theory.

One of the greatest contributions of Cantor was the discovery of ordinals and cardinals and the distinction between countable and uncountable sets of reals. This led him to the formulation of the continuum hypothesis, $\CH$, which states that the cardinality of the set of all real numbers is the smallest one conceivable in light of this, the smallest uncountable cardinal $\aleph_1$, cf. \cite{Cantor1877}. In 1904 Ernst Zermelo axiomatised set theory in a way conforming to mathematician's practise hitherto, cf. \cite{904Z0}. Subsequently, A. Fraenkel added the replacement scheme, cf. \cite{922F0}, thus yielding the system $\ZFC$. Somewhat later, Kurt G\"odel showed that the continuum hypothesis cannot be refuted within this system (provided that there is anything which cannot be derived within it), cf. \cite{940G0}. G\"odel conjectured, cf. \cite[Section 4]{947G0}, that neither can it be proved in it but only several decades later, Paul Cohen, at the origin an analyst just like Cantor, developed the method of forcing and could prove that this is indeed the case, cf. \cite{963C0}.

Nowadays it is common to subdivide analysis into real and complex analysis. The latter's theorems about its objects of study, holomorphic functions, revealed deep connections between analysis and geometry and found applications in various areas, among them number theory. One striking feature of the family of functions which are holomorphic on some domain of complex numbers opposite the family of those which are merely smooth on an interval of real numbers is the intertwinement of the local and global behaviour of holomorphic functions. Whereas two distinct functions may be both identical and infinitely often differentiable on an interval of real numbers, the situation in the complex domain is quite different, due to the famous ``identity theorem". According to it, for any two distinct functions holomorphic on some complex domain, the set of points where they agree is discrete (see Proposition~\ref{prop:constantaccum}).

As the complex plane is separable, no uncountable set of complex numbers is discrete. Therefore, any two distinct holomorphic functions can only agree on a countable set of points. Subsequently further theorems underscored the difference between the realms of entire functions on the one hand and smooth functions on the real number line on the other. In the middle of the nineteenth century, it emerged from work of Liouville, cf. \cite[Chapter 11]{Gray2015}, that all bounded entire functions are constant. Later, Nevanlinna developed the theory named after him, cf. \cite{Cherry2001}, one upshot of which is that for any two distinct entire functions $f$ and $g$ there can be at most four complex values $a$ for which the preimages of $\{ a \}$ are equal. Another interesting property that is purely of combinatorial nature and can be stated irrespective of the topological and algebraic structure of $\mathbb{C}$, is Picard's Little Theorem, namely the fact that any non-constant entire function can avoid at most one single value (see e.g. \cite[Theorem 16.22]{Rudin87}). 

The emerging picture of complex analysis in general and of entire functions in particular was one of strong general principles governing their behaviour. Against this backdrop, while writing his dissertation during the sixties of the last century, John E. Wetzel asked (cum grano salis) whether any family $\mathcal{F}$ of entire functions such that for each complex number $z$ the set $\{f(z) :  f \in \mathcal{F}\}$ is countable, must be countable itself. Dixon showed that, assuming the failure of the continuum hypothesis, this is indeed the case --- in fact, this is a corollary of the aforementioned identity theorem, cf. \cite{015GS0}. Shortly thereafter, Erd\H{o}s proved that not only does the negation of the continuum hypothesis imply this statement, it is equivalent to it, \cite{Erdoes}. In fact this result is one among many statements in various areas of mathematics proved to be equivalent to the continuum hypothesis, cf. \cite{934S0, 968B0}.
Towards the end of his paper, Erd\H{o}s asked whether the analogue statement resulting from replacing ``countable" by ``fewer than continuum many" can be proved without assuming the continuum hypothesis. Following a suggestion by Martin Goldstern we subsequently refer to a family of entire functions whose members everywhere assume fewer values than the family has members altogether as a \emph{Wetzel family} (see Definition~\ref{def:Wetzel}). In this terminology, Erd\H{o}s asked whether the existence of a Wetzel family is provable from $\ZFC$. One might also ask whether the continuum hypothesis is equivalent to the existence of a Wetzel family.

Not long after the development of forcing by Cohen, Solovay and Tennenbaum instigated the theory of iterated forcing and Tony Martin stated what became known as Martin's Axiom, or $\MA$ for short, a weakening of the continuum hypothesis. It does not prescribe a particular value for the cardinality of the continuum but it does for instance imply that its cardinality is regular. In many cases when something can be proved for countable sets within $\ZFC$, Martin's Axiom allows us to generalise this to sets with fewer than $2^{\aleph_0}$ elements. Meanwhile Erd\H{o}s' question remained unanswered.

The threads regarding Wetzel's problem were only picked up again in 2017 by Ashutosh Kumar and Saharon Shelah who answered both questions above in the negative, albeit in a slightly non-satisfactory way, see \cite{Kumar2017}. They showed that there is no Wetzel family in the side-by-side Cohen model and provided a model with a Wetzel family (and hence a continuum, cf. Lemma~\ref{lem:sizeWetzel}) of cardinality $\aleph_{\omega_1}$. The singularity of $\aleph_{\omega_1}$ is quite crucial in their argument and the result could not be generalized to other cardinals. Moreover, their model necessarily fails to satisfy Martin's Axiom.

It seems that after \cite{Kumar2017}, interest in Wetzel's problem has grown. We are aware of two more papers dealing with it since then, a formalisation of Erd\H{o}s' proof, \cite{Paulson}, and a proof that the continuum hypothesis implies the existence of sparse analytic systems, \cite{Cody}.\footnote{We have been informed that the authors of \cite{Cody} have been unaware of \cite{Kumar2017}, and were motivated rather by Erdős' original paper.} But no one yet addressed the open question which Kumar and Shelah ask at the end of \cite{Kumar2017}, of whether the existence of a Wetzel family is consistent with a continuum of cardinality $\aleph_2$. Erd\H{o}s' proof relied on the fact that any countable dense set of complex numbers is universal for countable sets via entire functions, that is to say that any countable set may be mapped into it via a non-constant entire function (see Definition~\ref{def:univ} and Proposition~\ref{prop:countableuniv}). In fact one finds a few papers from the sixties and seventies of the last century studying similar yet slightly stronger mapping properties, cf. \cite{967M1, 970BS0, 972BS0, 974SR0, 976NT0}. Kumar and Shelah observed that a Wetzel family would exist in a model of $2^{\aleph_0} = \aleph_2$ in which there is a set of cardinality $\aleph_1$ universal in the sense above for sets of cardinality $\aleph_1$ of complex numbers.

The main result of our paper is Theorem~\ref{thm:main}, that shows that starting from a model of the generalised continuum hypothesis and any cardinal $\kappa$ of uncountable cofinality, there is a cardinal and cofinality preserving forcing extension with a Wetzel family of size $\kappa$. In particular this completely solves Kumar and Shelah's open problem by showing that Wetzel families put no further restriction on the size of the continuum. Moreover, for regular $\kappa$ we can also force Martin's axiom. We also study the notion of universality from above and show that while $\MA$ precludes the existence of sufficiently universal sets, they can consistently exist while $2^{\aleph_0} = \aleph_2$.

While some basic knowledge of set theory is needed to understand the main results, a large part of the arguments (with an exception of those in Section~\ref{subsec:Baumgartner} and \ref{sec:properuniv}) is more analytic than set theoretic.

The paper is organized as follows. In the following first section we review some of the preliminaries in complex analysis and forcing that are used in the paper. In the next section, Section~\ref{sec:wetzelintro}, we introduce Wetzel families and universal sets and prove some $\ZFC$ results about them. Among other things, we show that Wetzel families must have cardinality $2^{\aleph_0}$, we provide a proof of Erd\H{o}s' result on countable dense sets and we show that universal sets imply the existence of Wetzel families. In Section~\ref{subsec:Baumgartner}, we show how to force a certain family of strongly almost disjoint functions that serves as a basic (and somewhat necessary) ingredient in the proof of the main result. In fact, it turns out that this solves \cite[Question 22]{Zapletal1997}. We also obtain some interesting additional results related to $\MA$ that shine some light on one of our open questions. This section can be read completely independently from the rest the paper. Section~\ref{sec:wetzelcontinuum} is the longest and contains the proof the main result, Theorem~\ref{thm:main}. In Section~\ref{sec:mauniv}, we show that universal sets do not exist under $\MA + \neg \CH$. As a corollary we obtain that the converse of Proposition~\ref{prop:univwetzel}, namely the statement that Wetzel families imply the existence of a universal set, does not hold. In the next Section, we then show that a universal set, as suggested by Kumar and Shelah, can consistently exist with continuum $\aleph_2$. This uses a proper forcing, based on some of the previous arguments, with pairs of models as side-conditions. We finish the paper with a list of open problems.  

\section{Preliminaries}

\subsection{Complex analysis}

Throughout the paper, $\mathbb{C}$ denotes the set of complex numbers. A function $f \colon \mathbb{C} \to \mathbb{C}$ is \emph{entire} if it is holomorphic on the domain $\mathbb{C}$, in other words, its complex derivative $f'$ exists in every point $z \in \mathbb{C}$. The set of entire functions will be denoted $\Hol$. 

\begin{prop}[see {\cite[Theorem 10.18]{Rudin87}}]\label{prop:constantaccum}
    Let $f,g \in \Hol$ and assume that the set $\{ z \in \mathbb{C} : f(z) = g(z) \}$ has an accumulation point. Then $f=g$.
\end{prop}

For $z \in \mathbb{C}$ and $\delta$ a positive real number, we let $$B_\delta(z) = \{ z' \in \mathbb{C} : \vert z - z' \vert < \delta \}$$ be the ball of radius $\delta$ around $z$. We also define the semi-norms $$\| f \|_\delta = \sup_{z \in B_\delta(0)} \vert f(z) \vert.$$ 
Recall that a sequence of functions $\bar f = \langle f_n : n \in \omega \rangle$ on $\mathbb{C}$ is said to converge uniformly on compact sets, if for every $\delta > 0$, $\bar f$ converges uniformily on $B_\delta(0)$, i.e. for every $\varepsilon > 0$, there is $N \in \omega$ so that for all $n_0,n_1 > N$, $$\| f_{n_0} - f_{n_1} \|_\delta < \varepsilon.$$ Among the most useful facts about entire functions that we will use is the following.

\begin{prop}[see {\cite[Theorem 10.28]{Rudin87}}]\label{prop:convergence}
Let $\langle f_n : n \in \omega\rangle$ be a sequence of entire functions that converges uniformily on every compact set. Then the pointwise limit $f$ of $\langle f_n : n \in \omega\rangle$ is entire.

Moreover, the sequence of derivatives $\langle f'_n : n \in \omega \rangle$ converges uniformily on every compact set to $f'$.
\end{prop}

\subsection{Forcing}

Here we review a few standard facts about forcing that we will find useful. We use standard forcing notation as used in the reference books \cite{Jech1997} or \cite{Kunen2011}.

\begin{lemma}[{\cite[Lemma V.3.9, V.3.10]{Kunen2011}}]\label{lem:ccctwostep}
    Let $\mathbb{P}*\dot{\mathbb{Q}}$ be a two step iteration. Then $\mathbb{P}*\dot{\mathbb{Q}}$ is ccc iff $\mathbb{P}$ is ccc and $\Vdash_{\mathbb{P}} \dot{\mathbb{Q}} \text{ is ccc}$.
\end{lemma}

Note of course, that if $\dot{\mathbb{Q}}$ is in fact of the form $\check{\mathbb{Q}}$ for some ground model forcing $\mathbb{Q}$, then also $\mathbb{P} *\dot{\mathbb{Q}}$ is ccc iff $\mathbb{P} \times \mathbb{Q}$ is.

\begin{definition}\label{def:subforcing}
    Let $\mathbb{P}, \mathbb{Q}$ be forcing notions. Then $\mathbb{P}$ is a subforcing of $\mathbb{Q}$ if $\mathbb{P} \subseteq \mathbb{Q}$ and the extension as well as the incompatibility relations agree. 
\end{definition}

\begin{definition}
    Let $M\subseteq V$ be a transitive model of $\ZF^-$ (possibly a proper class).\footnote{$\ZF^-$ is $\ZF$ without the Powerset Axiom.} Let $\mathbb{P} \in M$ be a subforcing of $\mathbb{Q}$. Then we write $\mathbb{P} \lessdot_M \mathbb{Q}$ to say that every predense set $E \in M$ of $\mathbb{P}$ is predense in $\mathbb{Q}$. We write $\mathbb{P} \lessdot \mathbb{Q}$ for $\mathbb{P} \lessdot_V \mathbb{Q}$ and say that $\mathbb{P}$ is a complete subforcing of $\mathbb{Q}$.
\end{definition}

\begin{lemma}[{\cite[Theorem 6.3]{SolovayTennenbaum}}]\label{lem:directccc}
    The iterative direct limit of ccc forcings is ccc. To be more precise, suppose $\langle \mathbb{P}_\delta : \delta \leq \alpha \rangle$ is a sequence of posets, $\alpha$ limit, so that 
    \begin{enumerate}
        \item for all $\gamma \leq \delta\leq \alpha$, $\mathbb{P}_\gamma \lessdot \mathbb{P}_\delta$, 
        \item for every limit $\delta \leq \alpha$, $\bigcup_{\gamma < \delta } \mathbb{P}_\gamma$ is dense in $\mathbb{P}_\delta$, 
        \item and for all $\delta < \alpha$, $\mathbb{P}_\delta$ is ccc.
    \end{enumerate}
    Then also $\mathbb{P}_\alpha$ is ccc.
\end{lemma}

\begin{lemma}\label{lem:comsub}
    Let $\mathbb{P}$ be a complete subforcing of $\mathbb{Q}$, $\dot{\mathbb{A}}$ a $\mathbb{P}$-name and $\dot{\mathbb{B}}$ a $\mathbb{Q}$-name for a forcing notion. Then $$\Vdash_{\mathbb{Q}} \dot{\mathbb{A}} \lessdot_{V^\mathbb{P}} \dot{\mathbb{B}}  \text{ iff } \mathbb{P} * \dot{\mathbb{A}} \lessdot \mathbb{Q}* \dot{\mathbb{B}}.$$
\end{lemma}

\begin{proof}
   It suffices to notice that a predense subset of $\mathbb{P}*\dot{\mathbb{A}}$ (and respectively of $\mathbb{Q}* \dot{\mathbb{B}}$), is precisely the same as a $\mathbb{P}$-name (or $\mathbb{Q}$-name) for a predense subset of $\dot{\mathbb{A}}$ (respectively $\dot{\mathbb{B}}$). For the direction from left to right, see also for example \cite[Lemma 13]{BrendleFischer}. 
\end{proof}

\section{Wetzel families and universal sets}\label{sec:wetzelintro}

\begin{definition}\label{def:Wetzel}
    A family $\mathcal{F} \subseteq \Hol$ of entire functions is called a \emph{Wetzel family} if for every $z \in \mathbb{C}$, $\vert \{ f(z) : f \in \mathcal{F} \} \vert < \vert \mathcal{F} \vert$.
\end{definition}

\begin{lemma}\label{lem:sizeWetzel}
If $\mathcal{F}$ is a Wetzel family, then $ \vert \mathcal{F} \vert = 2^{\aleph_0}$ and for every $\lambda < 2^{\aleph_0}$ and all but less than $2^{\aleph_0}$-many $z \in \mathbb{C}$, $\vert \{f(z) : f \in \mathcal{F} \} \vert \geq \lambda$.
\end{lemma}

\begin{proof}
    Suppose towards a contradiction that $\vert \mathcal{F} \vert < 2^{\aleph_0}$. For any distinct $f,g \in \Hol$, $X_{f,g} = \{ z \in \mathbb{C} : f(z) = g(z)\}$ does not have an accumulation point by Proposition~\ref{prop:constantaccum} and thus must be countable. Thus, if $X = \bigcup_{f \neq g \in \mathcal{F}} X_{f,g}$, then $$\vert  X\vert \leq \vert \mathcal{F} \times \mathcal{F} \times \omega \vert < 2^{\aleph_0}.$$
    In particular, there is some $z \in \mathbb{C} \setminus X$. But then for any distinct $f,g \in \mathcal{F}$, $f(z) \neq g(z)$ and so $\vert \{ f(z) : f \in \mathcal{F} \}\vert = \vert \mathcal{F} \vert$, contradicting that $\mathcal{F}$ is Wetzel.
    
    Now let $\lambda < 2^{\aleph_0}$ and $\mathcal{G} \subseteq \mathcal{F}$ have size $\lambda$. Then if $X' =  \bigcup_{f \neq g \in \mathcal{G}} X_{f,g}$, as before $\vert X' \vert < 2^{\aleph_0}$ and for all $z \in \mathbb{C} \setminus X'$, $\vert \{ f(z) : f \in \mathcal{F} \} \vert \geq \vert \{ f(z) : f \in \mathcal{G} \} \vert \geq \lambda$.
\end{proof}

This lemma in fact shows that a Wetzel family induces a quite non-trivial combinatorial object, especially when $2^{\aleph_0} > \aleph_2$. Namely, consider an enumeration $\langle z_\alpha : \alpha < \kappa \rangle$ of $\mathbb{C}$. For each $\alpha < \kappa$, there is a bijection $e_\alpha \colon \{f(z_\alpha) :f \in \mathcal{F} \}\to \mu_\alpha$ for some cardinal $\mu_\alpha< 2^{\aleph_0}$. In this way, we can think of each $f \in \mathcal{F}$ as the function $\sigma_f \in \prod_{\alpha < \kappa} \mu_\alpha$, where $\sigma_f(\alpha) = e_\alpha(f(z_\alpha))$. At the same time, the elements of $\mathcal{F}$ and thus of $\{\sigma_f : f \in \mathcal{F} \}$ have pairwise countable intersections. 

Under $\CH$, this type of almost disjoint family of functions can be obtained quite easily. For instance, for any $\alpha < \omega_1$, simply let $\sigma_\alpha$ be constantly $0$ below $\alpha$ and constantly equal to $\alpha$ above $\alpha$. 

Also this is not particularily hard when $2^{\aleph_0} = \aleph_2$. Whenever we have constructed $\langle \sigma_\beta : \beta < \alpha \rangle$ for some $\alpha < \omega_2$, we can find a single function $\sigma \colon \alpha \to \omega_1$ that has countable intersection with all $\sigma_\beta$ using a standard diagonalization argument. Then we simply let $\sigma_\alpha$ equal $\sigma$ below $\alpha$ and constantly equal to $\alpha$ above $\alpha$.

For larger continuum, the existence of such families becomes much less clear, as a consequence of the larger gap between the countable size of the pairwise intersections and the size of the family and their elements. In fact, in Section~\ref{sec:wetzelcontinuum} we will show how to force these types of families, based on a techinque by Baumgartner. This will be a key starting point for the construction of a Wetzel family by forcing.

As a small observation of independent interest let us mention:

\begin{lemma}
    A Wetzel family cannot consist only of polynomials.
\end{lemma}

\begin{proof}
    Suppose $\mathcal{F}$ is such a family. Since $\vert \mathcal{F} \vert = 2^{\aleph_0}$ and $2^{\aleph_0}$ has uncountable cofinality, we can assume that all polynomials in $\mathcal{F}$ have the same degree $n$. Then pick any $n+1$-many points $a_0, \dots, a_n \in \mathbb{C}$. The set $\{ f \restriction \{ a_0, \dots, a_{n} \} : f \in \mathcal{F} \}$ has size less than $2^{\aleph_0}$, because $\mathcal{F}$ is Wetzel. But each $f \in \mathcal{F}$ is uniquely determined by $f \restriction \{a_0, \dots, a_n \}$ and so $\mathcal{F}$ has small size as well, contradicting Lemma~\ref{lem:sizeWetzel}.
\end{proof}

\begin{definition}\label{def:univ}
    We call a set $Y \subseteq \mathbb{C}$, where $\vert Y \vert < 2^{\aleph_0}$, \emph{universal (for entire functions)} if for any $X \subseteq \mathbb{C}$ with $\vert X \vert < 2^{\aleph_0}$, there is a non-constant $f \in \mathcal{H}(\mathbb{C})$, such that $f(X) \subseteq Y$.
\end{definition}

\begin{prop}[Erd\H{o}s, {\cite{Erdoes}}]\label{prop:countableuniv}
    Assuming $\CH$, any countable dense set is universal. 
\end{prop}

Let us write the argument for sake of completeness. In a way, the forcing notions we will use later mimic this construction. 

\begin{proof}
    Let $Y \subseteq \mathbb{C}$ be countable dense and $X \subseteq \mathbb{C}$ be arbitrary countable and enumerated as $\langle z_n : n \in \omega \rangle$. We recursively construct a sequence $\langle f_n : n \in \omega \rangle$ of entire functions that converges uniformily on compact sets. Start by simply letting $f_0$ be constantly $0$. Next, if $f_n$ has been defined, consider $$g_{n, \xi}(z) = \xi\prod_{m < n} (z - z_m).$$ Note that the set of zeros of $g_{n,\xi}$ is exactly $\{ z_m : m < n\}$, when $\vert \xi \vert > 0$. Then there is $\delta > 0$, so that $$\| g_{n,\xi} \|_n < \frac{1}{2^n},$$ for every $\xi \in B_\delta(0)$. Since $Y$ is dense, there is some $\xi \in B_\delta(0)$ so that $$f_{n}(z_n) + g_{n, \xi}(z_n) \in Y \setminus f_n[\{ z_m : m < n \}].$$ Then let $f_{n+1} = f_n + g_{n,\xi}$. Finally, let $f$ be the limit of $\langle f_n : n \in \omega \rangle$. By Proposition~\ref{prop:convergence}, $f$ is entire. Clearly $f(X) \subseteq Y$ since $f_m(z_n)$ remains constant for $m > n$ and equals $f_{n+1}(z_n) \in Y$. Moreover $f$ is injective on $X$, so definitely non-constant.
\end{proof}

\begin{prop}\label{prop:univsuccessor}
    Let $Y$ be a universal set. Then $\vert Y \vert^+ = 2^{\aleph_0}$ and in particular the continuum is a successor cardinal.  
\end{prop}

\begin{proof}
    Suppose that there is $X \subseteq \mathbb{C}$ uncountable with $\vert Y \vert < \vert X \vert < 2^{\aleph_0}$. If $f \in \Hol$ and $f''X \subseteq Y$, by the pigeonhole principle, there is $y \in Y$ such that $\{ z \in X : f(z) = y \}$ is uncountable. But then this set has an accumulation point and by Proposition~\ref{prop:constantaccum}, $f$ is constant.
\end{proof}

\begin{prop}\label{prop:univwetzel}
    If there is a universal set there is also a Wetzel family.
\end{prop}

\begin{proof}
    Let $Y$ be universal and let $\langle z_\alpha : \alpha < \kappa \rangle$ enumerate $\mathbb{C}$. For each $\alpha < \kappa$, let $f_\alpha \in \Hol$ be non-constant such that $f_\alpha(\{ z_\beta : \beta < \alpha\}) \subseteq Y$. We claim that $\mathcal{F} = \{f_\alpha : \alpha < \kappa \}$ is a Wetzel family. 
    
    First of all, we note that $\vert \mathcal{F} \vert = \kappa$. By Proposition~\ref{prop:univsuccessor}, $\kappa$ is a successor and in particular regular. Thus, if $\vert \mathcal{F} \vert < \kappa$, there is an unbounded subset $S \subseteq \kappa$ so that $f_\alpha = f_\beta$, for all $\alpha, \beta \in S$. But then for any such $\alpha \in S$, $f_\alpha(\mathbb{C}) \subseteq Y$. This would imply that $f_\alpha$ is constant, as in the proof of Proposition~\ref{prop:univsuccessor}.

    Now let $z_\alpha \in \mathbb{C}$ be arbitrary. Then $\{ f(z_\alpha) : f \in \mathcal{F} \} \subseteq \{f_\beta(z_\alpha) : \beta \leq \alpha \} \cup Y$ and the right-hand-side has size less than $\kappa$. 
\end{proof}

\begin{cor}[Erd\H{o}s, {\cite{Erdoes}}]\label{cor:WetzelCH}
    There is a Wetzel family under $\CH$.
\end{cor}

\section{Strongly almost disjoint functions and Baumgartner's thinning-out forcing}\label{subsec:Baumgartner}

This section can be skipped entirely if one wants to pass directly to the proof of the main result. The main purpose is to prove Proposition~\ref{prop:Baumgartner} below which is used in the setup of our main forcing construction. It is an adaptation of Baumgartner's ``thining-out" technique to obtain certain types of almost disjoint families (see \cite{Baumgartner76}). To be more precise, we show how to obtain the type of family of functions necessitated by a Wetzel family, as shown in Section~\ref{sec:wetzelintro}, where pairwise intersections are finite. Interestingly, a different type of strongly almost disjoint family was also used in \cite{Kumar2017} to obtain a Wetzel family with continuum $\aleph_{\omega_1}$. It is unclear to us how this relates to our argument.

\begin{prop}\label{prop:Baumgartner}
    (GCH) Let $\kappa$ be an infinite cardinal of uncountable cofinality and for every $\alpha < \kappa$, let $\mu_\alpha = \max(\vert \alpha \vert, \aleph_0)$. Then there is a cardinal and cofinality preserving forcing extension of $V$ where $2^{\aleph_0} = \kappa$ and there is  $\langle \sigma_\alpha : \alpha < \kappa \rangle$, such that for all $\alpha < \beta < \kappa$, 
    \begin{enumerate}
        \item $\sigma_\alpha \in \prod_{\xi < \kappa} \mu_\xi$, 
        \item $\vert \sigma_\alpha \cap \sigma_\beta \vert < \omega$.
    \end{enumerate}

If $\kappa$ is regular, we additionally have that $\vert H(\kappa) \vert = \kappa$.
\end{prop}

\begin{proof}

    Let $S = \bigcup_{\xi \in [\omega,\kappa)} \{ \xi \} \times \mu_\xi \subseteq [\omega, \kappa) \times \kappa$ and let $K$ be the set of regular cardinals $\leq \kappa$. For every $\lambda \in K$ consider the forcing $\mathbb{P}_\lambda$ consisting of partial functions $p \colon \kappa \to [S]^{<\lambda}$, such that 
    \begin{enumerate}
        \item $\vert \dom p \vert < \lambda$,
        \item for all $\alpha, \beta \in \dom p$, $\proj p(\alpha) = \proj p(\beta)$, i.e. $p(\alpha)$ and $p(\beta)$ have the same projection to the first coordinate, 
        \item and for all $\alpha \in \dom p$ and $\xi \in \proj p(\alpha) \cap \lambda$, $\vert p(\alpha) \cap (\{\xi \} \times \mu_\xi) \vert = \mu_\xi$.
    \end{enumerate}    
We will simply write $\proj p$ to denote $\proj p(\alpha)$, for any $\alpha \in \dom p$, and if $\dom p = \emptyset$, we let $\proj p = \emptyset$. 
    
    For any $p,q \in \mathbb{P}_\lambda$, $q \leq p$ iff $\dom p \subseteq \dom q$ and for any distinct $ \alpha, \beta \in \dom p$, 

    \begin{enumerate}
        \item $p(\alpha) \subseteq q(\alpha)$,
        \item and $q(\alpha) \cap q(\beta) \cap ([\lambda^-, \kappa) \times \kappa) = p(\alpha) \cap p(\beta) \cap ([\lambda^-, \kappa) \times \kappa)$,
    \end{enumerate}
where $\lambda^-$ is the predecessor of $\lambda$, if $\lambda$ is a successor cardinal, or $\lambda^- = \lambda$, if $\lambda$ is a limit cardinal.

   Whenever $G$ is $\mathbb{P}_\lambda$-generic, let $S_\alpha = \bigcup_{p \in G} p(\alpha)$ for every $\alpha < \kappa$. Then it is not hard to check that every vertical section of $S_\alpha$ is non-empty (see more below) and for any $\alpha < \beta$, $$\vert S_\alpha \cap S_\beta \cap ([\lambda^-, \kappa) \times \kappa)\vert < \lambda.$$ In particular, when $\lambda$ is a successor cardinal, $$\vert S_\alpha \cap S_\beta \vert < \lambda.$$ In fact, for any $\xi \in [\omega, \lambda^-)$, the section with index $\xi$ of $S_\alpha$ equals $\mu_\xi$. The reason why we include these sections is purely notational.

   It is clear that $\mathbb{P}_\lambda$ is $<\lambda$-closed. Ideally we would like to force with $\mathbb{P}_\omega$, since if we then choose $\sigma_\alpha$, such that $(\omega + \xi, \sigma_\alpha(\xi)) \in S_\alpha$, for every $\xi < \kappa$, (1) and (2) of the proposition are satisfied. But $\mathbb{P}_\omega$ is far from being ccc. To circumvent this we use Baumgartner's thinning out trick. 

   Let $\mathbb{P} \subseteq \prod_{\lambda \in K} \mathbb{P}_\lambda$ consist of all $\bar p = \langle p_\lambda : \lambda \in K \rangle$ such that for $\lambda' < \lambda$, $\dom p_{\lambda'} \subseteq \dom p_{\lambda}$ and for any $\alpha \in \dom p_{\lambda'}$, $p_{\lambda'}(\alpha) \subseteq p_{\lambda}(\alpha)$. When $G$ is $\mathbb{P}$-generic, we obtain the sets $S^\lambda_\alpha = \bigcup_{\bar p \in G} p_\lambda(\alpha)$ for every $\lambda \in K$ and it is very easy to see again that $$\vert S^\lambda_\alpha \cap S^\lambda_\beta \cap([\lambda^-, \kappa) \times \kappa)\vert < \lambda,$$ for $\alpha < \beta < \kappa$. Let us check that all vertical sections of $S^\lambda_\alpha$ are non-empty. To this end let $\bar p \in \mathbb{P}$ be arbitrary. 

  \begin{claim}
      There is $\bar q \leq \bar p$ so that $\alpha \in \dom q_\lambda$.
  \end{claim}

  \begin{proof}
      If for all $\nu \in K$, $\alpha \notin \dom p_\nu$, we let $\dom q_\nu = \dom p_\nu \cup \{\alpha \}$, $q_\nu \restriction \dom p_\nu = p_\nu \restriction \dom p_\nu$ and $$q_\nu(\alpha) = \{ (\xi, 0) : \xi \in \proj p_\nu \setminus \nu \} \cup \bigcup_{\xi \in \proj p_\nu \cap \nu} \{\xi \} \times \mu_\xi,$$ for every $\nu \in K$. Otherwise, there is a minimal $\mu \in K$ so that $\alpha \in \dom p_\mu$. In this case, pick the least $a_\xi$ so that $(\xi, a_\xi) \in p_\mu(\alpha)$, for every $\xi \in \proj p_\mu$. Then for each $\nu \in K$, we let $q_\nu$ extend $p_\nu$ as before, but with $$ q_\nu(\alpha) = p_\nu(\alpha) \cup \{ (\xi, a_\xi) : \xi \in \proj p_\nu \setminus \nu \} \cup \bigcup_{\xi \in \proj p_\nu \cap \nu} \{\xi \} \times \mu_\xi.$$
  \end{proof}

  Now let $\xi \in [\omega, \kappa)$ be arbitrary. 

\begin{claim}
      There is $\bar q \leq \bar p$ so that $\xi \in \proj q_\lambda$.
  \end{claim}

  \begin{proof}
      If $\xi < \lambda^-$, simply extend every single $p_\nu(\beta)$, for $\nu \geq \lambda$ and $\beta \in \dom p_\nu$, by adding $\{ \xi\} \times \mu_\xi$. This works since pairwise intersections occurring below $\nu^- \geq \lambda^-$ do not matter when extending $p_\nu$.

      If $\xi \in [\lambda^-, \lambda)$, $\lambda$ is a successor and $\mu_\xi = \lambda^-$. Then we can assume already that for every $\nu \geq \lambda^+$ and $\beta \in \dom p_\nu$, $\{\xi\} \times \mu_\xi \subseteq p_\nu(\beta)$, using the same procedure as before. Since $\vert \dom p_\lambda \vert < \lambda$ and thus $\vert \dom p_\lambda \vert \leq \lambda^-$, it is easy to find pairwise disjoint sets $X_\beta \subseteq \mu_\xi = \lambda^-$ of size $\mu_\xi = \lambda^-$, for all $\beta \in \dom p_\lambda$. Simply extend each $p_\lambda(\beta)$ by adding $\{\xi\} \times X_\beta$.  

      Finally, if $\xi \geq \lambda$, we can assume already that $\xi \in \proj p_\nu$, for $\nu = \vert \xi \vert^+$, using what we have just shown. Then we can easily find pairwise distinct $a_\beta$, so that $(\xi, a_\beta) \in p_\nu(\beta)$, for all $\beta \in \dom p_\nu$, since $\vert \dom p_\nu \vert \leq \nu^- = \vert \xi \vert = \mu_\xi$ and $\vert p_\nu(\beta) \cap ( \{\xi \} \times \mu_\xi) \vert = \mu_\xi$, for every $\beta \in \dom p_\nu$. Now simply extend all $p_{\nu'}(\beta)$, for $\nu' \in K \cap [\lambda, \nu)$ and $\beta \in \dom p_{\nu'}$, by adding in $(\xi, a_\beta)$.
  \end{proof}
  
 The $\sigma_\alpha$ as defined above, for $S_\alpha = S_\alpha^\omega$, are then as required.

\begin{claim}\label{claim:Baumpreserv}
    $\mathbb{P}$ preserves all regular cardinals.
\end{claim}

\begin{proof}
    Let $\lambda \in K$. We will show how to factor $\mathbb{P}$ into a two step iteration of a $<\lambda^+$-closed and a $\lambda^+$-cc forcing. Let $\mathbb{P}_0 = \{ \bar p \restriction [\lambda^+, \kappa] : \bar p \in \mathbb{P} \}$ and note that $\mathbb{P}_0$ is $<\lambda^+$-closed and thus doesn't add sequences of length $\lambda$.\footnote{When $\lambda = \kappa$, then $\mathbb{P}_0$ is simply the trivial forcing.} Let $G_0$ be $\mathbb{P}_0$-generic over $V$. Let $S_\alpha = S_\alpha^{\lambda^+}$, for $\alpha < \kappa$, be defined as before, i.e. $$S_\alpha = \bigcup_{\bar p \in G_0} p_{\lambda^+}(\alpha).$$

    Let $\mathbb{P}_1$ consist of all $\bar p \restriction \lambda^+$, for $\bar p \in \mathbb{P}$, where $$p_{\lambda}(\alpha) \subseteq S_\alpha,$$ for each $\alpha \in \dom p_{\lambda}$. Then it is easy to verify that if $G_1$ is $\mathbb{P}_1$-generic over $V[G_0]$, $V[G_0][G_1]$ is a $\mathbb{P}$-generic extension of $V$.  Work in $V[G_0]$ and suppose that $\langle \bar p^i : i < \lambda^+ \rangle$ is an antichain in $\mathbb{P}_1$. 
    
    Note that for any $i < \lambda^+$, there is $\delta < \lambda$ so that $\dom p^i_{\lambda'}$ as well as $p^i_{{\lambda'}}(\alpha)$ are the same for all $\lambda' \in [\delta, \lambda) \cap K$ and fixed $\alpha \in \dom p^i_{\lambda}$, since these sets grow and are subsets of $\dom p^i_\lambda$ and $p^i_\lambda(\alpha)$ respectively which have size strictly less than $\lambda$.\footnote{Of course, if $\lambda$ is a successor cardinal this is trivial as we may choose $\delta$ such that $K \cap [\delta, \lambda)$ is empty. Otherwise, recall that $\lambda$ is regular.} We may assume that this $\delta$ is the same for all $i < \lambda^+$ and let $$D^i = \dom p^i_{{\lambda'}},$$ for any, equivalently all $\lambda' \in [\delta, \lambda) \cap K$, if this exists.\footnote{ If $K \cap [\delta, \lambda)$ is empty, $D^i$ is left undefined as it will be irrelevant.}
    
    Since $\lambda^{<\lambda}  =\lambda$, a $\Delta$-system argument lets us assume that there are sets $D$ and $E$ such that $D = D^i \cap D^j$ and $\dom p^i_\lambda \cap \dom p^j_\lambda = E$, for all $i<j < \lambda^+$. Note that we can also do the same for all $\lambda' \in \delta \cap K$ simultaneously, as $\delta < \lambda$. To be more precise we can assume $D_{\lambda'} = \dom p^i_{\lambda'} \cap \dom p^j_{\lambda'}$, for all $\lambda' \in \delta \cap K$ and  $i < j < \lambda^+$. Moreover, using another $\Delta$-system argument we can assume that $R^i_\lambda \cap R^j_\lambda = R$, for some fixed $R$ and all $i < j < \lambda^+$, where $R_\lambda^i = \bigcup_{\alpha \in \dom p^i_\lambda} p^i_{\lambda}(\alpha)$.

    Now comes the heart of the thinning out argument. Let $i < \lambda^+$ be arbitrary, then note that for any distinct $\alpha, \beta$ and any $\lambda'$, $p^i_{\lambda'}(\alpha) \cap S(\beta)$ is a $< \lambda$ sized subset of $S(\alpha) \cap S(\beta)$, which has size at most $\lambda$. Also $R$ has size $<\lambda$.~
    Thus we find $i <  j < \lambda^+$ such that $$p^i_{\lambda'}(\alpha) \cap S(\beta) = p^j_{\lambda'}(\alpha) \cap S(\beta)$$ and $$p^i_{\lambda'}(\alpha) \cap R = p^j_{\lambda'}(\alpha) \cap R,$$ for all $\lambda' \in \delta \cap K$ and distinct $\alpha, \beta \in D_{\lambda'}$, also for any, equivalently all, $\lambda' \in [\delta, \lambda) \cap K$ and distinct $\alpha, \beta \in D$, and for $\lambda' = \lambda$ and distinct $\alpha, \beta \in E$. It is straightforward to check that $\bar p^i$ and $\bar p^j$ are compatible. 

    Now let us check that all regular cardinals $\theta$ are preserved. Suppose that in $V^\mathbb{P}$, $\cf(\theta) = \mu < \theta$. If $\mu \geq \kappa$, then either $\kappa$ is regular and we have shown above that $\mathbb{P}$ is $\kappa^+$-cc, so in particular $\theta$-cc. Or $\kappa$ is singular, $\mu \geq \kappa^+$ and $\vert \mathbb{P}\vert \leq \kappa^+$, so $\mathbb{P}$ is $\kappa^{++}$-cc and also $\theta$-cc. If $\mu < \kappa$, $\mathbb{P}$ is the iteration of a $\mu^+$-closed and a $\mu^+$-cc forcing and $\theta \geq \mu^+$.
\end{proof}

It is clear that in $V^{\mathbb{P}}$, $2^{\aleph_0} \geq \kappa$ since $\{ \sigma_\alpha \restriction \omega : \alpha < \kappa \}$ is an almost disjoint set of functions from $\omega$ to $\omega$. In the other direction, note that $2^\lambda \leq \kappa$ for any regular $\lambda < \cf(\kappa)$. Namely $\mathbb{P}$ is $\mathbb{P}_0 * \dot{\mathbb{P}}_1$, where $\mathbb{P}_0$ doesn't add any new sequences of length $\lambda$ and ${\mathbb{P}}_1$ has size at most $\kappa$ and is $\lambda^+$-cc. A counting of names argument then shows that $2^{\lambda} \leq \kappa$. In particular $2^{\aleph_0} = \kappa$ and $ \vert H(\kappa) \vert = 2^{<\kappa} = \kappa$, when $\kappa$ is regular.   
\end{proof}

\begin{cor}
    The answer to \cite[Question 22]{Zapletal1997} is positive. Namely, assuming $\GCH$, it possible to add arbitrarily large strongly almost disjoint families of functions in $\aleph_\omega^{\aleph_{\omega+1}}$ without collapsing cardinals.
\end{cor}

The following is of further interest, especially for Question~\ref{quest:PFAMA} at the end of the paper. In the second paragraph after the proof of Lemma~\ref{lem:sizeWetzel}, we have shown using a simple construction that:
\begin{prop}
    $2^{\aleph_0} = \aleph_2$ implies that there is $\langle \sigma_\alpha : \alpha < \omega_2 \rangle$ so that for any $\alpha < \beta < \omega_2$, $\sigma_\alpha \in ~^{\omega_2}\omega_1$ and $\vert \sigma_\alpha \cap \sigma_\beta \vert < \omega_1$.
\end{prop}

In fact, $\MA + 2^{\aleph_0} = \aleph_2$ (in particular $\PFA$) implies that we can even assume finite intersections, and thus the conclusion of Proposition~\ref{prop:Baumgartner} holds.

\begin{prop}
    $\MA + 2^{\aleph_0} = \aleph_2$ implies that there is $\langle \sigma_\alpha : \alpha < \omega_2 \rangle$ so that for any $\alpha < \beta < \omega_2$, $\sigma_\alpha \in ~^{\omega_2}\omega_1$ and $\vert \sigma_\alpha \cap \sigma_\beta \vert < \omega$.
\end{prop}

\begin{proof}
    We recursively construct a sequence as above, but where each $\sigma_\alpha$ is an element of $\prod_{\beta < \omega_2} \max(\omega_1, \beta + 1)$. This clearly makes no difference. We also ensure that $\sigma_\alpha$ always constantly maps to $\alpha$ on $[\alpha,\omega_2)$. Suppose that $\langle \sigma_\beta : \beta < \alpha \rangle$ has been constructed. Using a simple diagonalization argument, we find a function $S\colon \alpha \to [\omega_1]^\omega$, so that for any $\beta < \alpha$, $\{ \xi < \alpha : \sigma_\beta(\xi) \in S(\xi) \}$ is countable. Now let $\mathbb{P}$ be the natural poset with finite conditions adding a function $\sigma \in \prod_{\xi < \alpha} S(\xi)$ that has finite intersection with each $\sigma_\beta$, $\beta < \alpha$. To be more precise, $\mathbb{P}$ consists of pairs $(s,w)$ where $s$ is a finite partial function $s \in \prod_{\xi \in \dom(s)} S(\xi)$ and $w \in [\alpha]^{<\omega}$. A condition $(t,u)$ extends $(s,w)$ if $s \subseteq t$, $w \subseteq u$ and for every $\beta \in w$, $(t \setminus s) \cap \sigma_\beta = \emptyset$. Once we have shown that $\mathbb{P}$ is ccc, we can apply $\MA$ to find $\sigma$. Then simply let $\sigma_\alpha \restriction \alpha = \sigma$ and $\sigma_\alpha \restriction [\alpha, \omega_2)$ constantly equal $\alpha$.
    
    So suppose that $\langle (s_\delta, w_\delta) : \delta < \omega_1 \rangle$ is an uncountable antichain in $\mathbb{P}$. Without loss of generality we can assume that $\langle \dom(s_\delta) : \delta < \omega_1 \rangle$ forms a $\Delta$-system with root $r$ and that $s_\delta \restriction r = s$ and $\vert \dom(s_\delta) \vert = n$, for all $\delta$ and some fixed $s$ and $n$. Also we may assume that $\langle w_\delta : \delta < \omega_2 \rangle$ is a $\Delta$-system with root $w$. Since $\langle \dom(s_\delta) \setminus r : \delta < \omega_1 \rangle$ is pairwise disjoint and for each $\sigma_\beta$, $\sigma_\beta \cap \bigcup_{\xi < \alpha} \left(\{\xi \} \times S(\xi)\right)$ is countable, there is a large enough $\gamma \in [\omega, \omega_1)$, so that for every $\delta \in [\gamma, \omega_1)$, 
        $$s_{\delta} \restriction (\dom(s_{\delta}) \setminus r) \cap \sigma_\beta = \emptyset,$$ for all $\beta \in \bigcup_{i \in \omega} w_i$. Note that this means that $(s_i \cup s_{\delta}, w_i \cup w_{\delta} ) \leq (s_i, w_i)$, for every $i \in \omega$. Thus the only way in which $(s_{i}, w_{i})$ and $(s_\delta, w_\delta)$ can be incompatible, is if there is $\beta \in w_{\delta} \setminus w$ and $\xi \in \dom(s_i) \setminus r$, so that $$\sigma_\beta(\xi) = s_i(\xi).$$
        For any $\delta \geq \gamma$, let's define functions $\xi_{\delta}$ and $\beta_\delta$ with domain $\omega$ so that $\xi_\delta(i) \in \dom(s_i) \setminus r$ is a witness $\xi$ as above for $\beta_\delta(i) \in w_{\delta} \setminus w$. Since for all $i \in \omega$, $\dom(s_i) \setminus r$ has finite size $\leq n$, there must be $\delta_0 < \delta_1$ and an infinite $x \subseteq \omega$ so that $$\xi_{\delta_0} \restriction x = \xi_{\delta_1} \restriction x.$$ Since $w_{\delta_0} \setminus w$ and $w_{\delta_1} \setminus w$ are finite, we find an infinite $y \subseteq x$ so that both $\beta_{\delta_0} \restriction y$ and $\beta_{\delta_1} \restriction y$ are constant, say with values $\beta^0$ and $\beta^1$. Now $\beta^0 \neq \beta^1$, since $w_{\delta_0} \setminus w$ and $w_{\delta_1} \setminus w$ are disjoint. But then $\sigma_{\beta^0} \cap \sigma_{\beta^1}$ is infinite, yielding a contradiction.
\end{proof}

\section{Wetzel families with arbitrary continuum and MA}\label{sec:wetzelcontinuum}

In this section we prove our main result that Wetzel families can coexist with arbitrary values of the continuum and in combination with Martin's Axiom.

\subsection{Adding entire functions}

\begin{definition}\label{def:Q}
 The poset $\mathbb{Q}$ consists of all conditions $$p = (a_p, f_p, \varepsilon_p, m_p),$$ where $a_p \in [\mathbb{C}]^{<\omega}$, $f_p\in \mathcal{H}(\mathbb{C})$, $\varepsilon_p$ is a positive rational number and $m_p \in \omega$. A condition $q$ extends $p$ iff $a_p \subseteq a_q$, $f_q \restriction a_p = f_p \restriction a_p$, $\varepsilon_q \leq \varepsilon_p$, $m_p \leq m_q$ and $\| f_q - f_p \|_{m_p} \leq \varepsilon_p - \varepsilon_q.$
\end{definition}

\begin{definition}
Let $H \colon X \to \mathcal{P}(\mathbb{C})$, for some $X \subseteq \mathbb{C}$. Then we define $$\QH = \{ p \in \mathbb{Q} : a_p \subseteq X \wedge \forall z \in a_p (f_p(z) \in H(z)) \}.$$
\end{definition}

Note that the notion of incompatibility of conditions $p,q \in \QH$ isn't dependent on $H$. Namely, if $r$ extends $p$ and $q$ in $\mathbb{Q}$, then $(a_p \cup a_q, f_r, \varepsilon_r, m_r)\in \QH$ and also extends $p$ and $q$. In other words, $\QH$ is a subforcing of $\mathbb{Q}$ (see Definition~\ref{def:subforcing}). For most considerations it is also not relevant in which transitive model of set theory $M$ we evaluate the definition of $\mathbb{Q}(H)$, as long as $H \in M$. 

\begin{lemma}\label{lem:implicit}
    Let $M\subseteq V$ be a transitive model of $\ZF^-$(possibly a proper class) and $H \in M$ be a partial function from $\mathbb{C}$ to $\mathcal{P}(\mathbb{C})$. Then $\mathbb{Q}(H)^M$ is a dense subforcing of $\mathbb{Q}(H)^V$.
\end{lemma}
\begin{proof}
Note that conditions $p$ such that $f_p$ is a polynomial in coefficients in the field generated by $\dom H \cup \bigcup_{z \in \dom H} H(z) \subseteq M$ form a dense subposet of $\mathbb{Q}(H)^V$. Namely, if $ \| f - g \|_m = \delta < \varepsilon$, and $f\restriction a = g \restriction a$, $(a, g, \varepsilon - \delta, m) \leq (a,f,\varepsilon,m)$.
\end{proof}

In particular, any iteration of the form $\mathbb{Q}(H_0) * \dot{\mathbb{Q}}(H_1)$, where $H_0, H_1$ are in the ground model, is equivalent to the product $\mathbb{Q}(H_0) \times \mathbb{Q}(H_1)$ and the ccc of the iteration is equivalent to that of the product. This will be used at least implicitly in several arguments. 

\begin{lemma}\label{lem:genericfunction}
Suppose that $H(z)$ is dense in $\mathbb{C}$, for every $z\in\dom H$. Then $\QH$ generically adds an entire function $f$ such that $f(z) \in H(z)$ for every $z \in \dom H$. \end{lemma}

\begin{proof}
 Let $G$ be $\QH$-generic over $V$. For any $n \in \omega$, the set $D_n = \{ p \in \QH: \varepsilon_p < \frac{1}{n} \wedge m_p > n \}$ is clearly dense open. Moreover, for any $\xi \in \mathbb{C}$ and any $p \in \QH$ consider $$f_\xi(z) = f_p(z) + \xi \prod_{y \in a_p}(z-y).$$ Note that $f_\xi(y) = f_p(y)$, for every $y \in a_p$. Let $z \in \dom H \setminus a_p$ be arbitrary. Since $H(z)$ is dense, we can easily find a small enough $\xi$ so that $$\delta := \| f_p - f_\xi \|_{m_p} < \varepsilon_p$$ and $f_\xi(z) \in H(z)$. Then $(a_p \cup \{z\}, f_\xi, \varepsilon_p - \delta, m_p) \leq p$. This shows that $E_z = \{ q \in \QH : z \in a_q \}$ is dense open. We claim that for any sequences $\langle p_n : n \in \omega \rangle$ and $\langle q_n : n \in \omega \rangle$, with $p_n,q_n \in D_n \cap G$, $\langle f_{p_n}: n \in \omega\rangle$ and $\langle f_{q_n} : n \in \omega \rangle$ converge uniformly on compact sets to the same function $f \in \Hol$. To see this, use Proposition~\ref{prop:convergence} and notice that when $p,q \in D_n \cap G$ are arbitrary, there is $r \leq p,q$ and thus $$ \| f_p - f_q \|_{n} \leq  \| f_p - f_r \|_n + \| f_r - f_q \|_{n} < \frac{1}{n} + \frac{1}{n}.$$

For any $z \in \dom H$, we can find a decreasing sequence $\langle p_n : n \in \omega \rangle$ such that $p_n \in D_n \cap E_z \cap G$, for every $n$. Then $f(z) = \lim_{n \to \infty} f_{p_n}(z) = f_{p_0}(z) \in H(z)$. \end{proof}

Let us make the following interesting observation that will somewhat elucidate the necessity of the approach taken in the proof of the main result. 

\begin{lemma}\label{lem:nonccc}
    Let $\dom H$ be uncountable and suppose that there is an entire $f$ such that $f(z) \in H(z)$, for every $z \in \dom H$. Then $\mathbb{Q}(H)$ is not ccc.
\end{lemma}

\begin{proof}
Let $n \in \omega$ be such that $B_n(0) \cap \dom H$ is uncountable and let $\varepsilon> 0$ be so that $\| \RE(f') \|_n < \varepsilon$. For any $z_0 \in B_n(0) \cap \dom H$, define $$f_{z_0}(z) = 2\varepsilon(z- z_0)+ f(z_0).$$ Note that there is $\delta > 0$ and $m \ge n$ so that whenever $\| g - f_{z_0} \|_m < \delta$, then $\| g' - 2 \varepsilon \|_n < \varepsilon$.\footnote{For instance, this follows easily from Proposition~\ref{prop:convergence}. This part of the argument strongly depends on the special geometry of holomorphic functions. The statement is clearly not true for functions that are merely infinitely often differentiable.}
Let $p_{z_0} = (\{z_0 \}, f_{z_0}, \delta, m)$. We claim that $\{ p_{z_0} : z_0 \in B_n(0) \cap \dom H \}$ is an antichain. Namely, suppose that $z_0, z_1 \in B_n(0) \cap \dom H$ are arbitrary and that $r \leq p_{z_0}, p_{z_1}$. Then $\| f_r - f_{z_0} \|_m < \delta$ and so $\|f'_r - 2\varepsilon \|_n < \varepsilon$. In particular, for any $z \in B_n(0)$, $\RE(f_r'(z)) > \varepsilon$. At the same time, by the complex mean value theorem (see e.g. \cite[Theorem 2.2]{EvardJafari}), $$\varepsilon < \RE\left(\frac{f_r(z_0) - f_r(z_1)}{z_0 - z_1}\right)  = \RE\left(\frac{f(z_0) - f(z_1)}{z_0 - z_1}\right) < \varepsilon,$$ which poses a contradiction.
\end{proof}

It would be interesting to obtain some sort of converse to Lemma~\ref{lem:nonccc}. For instance, suppose that $H$ only maps to countable sets. Does the non-ccc of $\mathbb{Q}(H)$ imply at least that there is a Borel function $f$, with $f(z) \in H(z)$ for uncountably many $z \in \dom H(z)$?

\begin{cor}
    Let $\dom H$ be uncountable and suppose that $H(z)$ is dense in $\mathbb{C}$, for every $z\in\dom H$. Then $\mathbb{Q}(H) \times \mathbb{Q}(H)$ is not ccc.
\end{cor}

\begin{proof}
Either $\mathbb{Q}(H)$ is already not ccc, or $\dom H$ is preserved to be uncountable and by Lemma~\ref{lem:implicit} and ~\ref{lem:genericfunction} we are in the situation of Lemma~\ref{lem:nonccc} after forcing with $\mathbb{Q}(H)$ once. Thus $\Vdash_{\mathbb{Q}(H)} \text{``$\mathbb{Q}(H)$ is not ccc"}$ and by Lemma~\ref{lem:ccctwostep}, $\mathbb{Q}(H) * \dot{\mathbb{Q}}(H) \cong \mathbb{Q}(H) \times \mathbb{Q}(H)$ is not ccc.
\end{proof}

Thus the forcings $\QH$ can generally not be recycled in a ccc construction. If one wants to add another entire function, one has to pass to a new $H$.

 \subsection{Tools for the successor step}

\begin{lemma}\label{lem:interbound}
Let $l,m\in \omega$ and $K \subseteq \mathbb{C}^{l+1}$ be compact, such that every element of $K$ is one-to-one. Then there is $L > 0$ such that for any $\bar z \in K$, there is $g \in \mathcal{H}(\mathbb{C})$ with $\| g \|_m < L$, $g(z_i) = 1$ for every $i < l$ and $g(z_l) = 0$. 
\end{lemma}

\begin{proof}
Consider an interpolation formula such as 

    $$g(\bar z, z) = \sum_{i < k} \frac{(z-z_k)}{(z_i - z_k)} \prod_{\substack{j < k \\ j \neq i}} \frac{(z- z_j)}{(z_i - z_j)}, $$ and simply note that $\bar z \mapsto g(\bar z, \cdot)$ is a continuous map from $K$ to $\mathcal{H}(\mathbb{C})$ in the norm $\| \cdot \|_m$. The claim follows from the compactness of $K$.
\end{proof}

 \begin{lemma}
     Let $H_0, \dots, H_n$ be such that $\mathbb{Q}(H_0) \times \dots \times \mathbb{Q}(H_n)$ is ccc. Let $z \in \mathbb{C}$ be arbitrary and $H_0' \supseteq H_0$ where $\dom H_0' = \dom H_0 \cup \{ z\}$ and $H'_0(z)$ is countable. Then $\mathbb{Q}(H'_0) \times \mathbb{Q}(H_1) \times \dots \times \mathbb{Q}(H_n)$ is ccc. 
 \end{lemma}

 \begin{proof}
     Suppose towards a contradiction that $\langle \bar p_\alpha : \alpha < \omega_1 \rangle$ is an antichain in $\mathbb{Q}(H'_0) \times \mathbb{Q}(H_1) \times \dots \times \mathbb{Q}(H_n)$. Then we may assume without loss of generality that for every $\alpha < \omega_1$, $z \in a_{p_\alpha(0)}$ and $f_{p_\alpha(0)}(z) = y$, for some fixed $y \in H_0'(z)$. Otherwise, we find an uncountable antichain in $\mathbb{Q}(H_0) \times \dots \times \mathbb{Q}(H_n)$. Furthermore, we may assume that $\varepsilon_{p_\alpha(0)} = \varepsilon$, $m_{p_\alpha(0)} = m$, $\vert a_{p_\alpha(0)} \vert = l + 1$ and $a_{p_\alpha(0)} \setminus \{ z \}$ is enumerated by $\bar z_\alpha = \langle z_{\alpha,i} : i < l \rangle$, for every $\alpha$ and some fixed $\varepsilon$, $m$ and $l$. Even more, we can assume that $\|f_{p_\alpha(0)} - f_{p_\beta(0)}\|_m < \frac{\varepsilon}{2}$, for every $\alpha, \beta < \omega_1$. 
     
     Then there is some $\beta < \omega_1$, such that $\bar z_\beta$ is an $\omega_1$-accumulation point of \linebreak $\{ \bar z_\alpha : \alpha < \omega_1 \}$ in $\mathbb{C}^l$, in the sense that for any open neighborhood of $\bar z_\beta$, there are uncountably many $\alpha$ with $\bar z_\alpha$ in said neighborhood.\footnote{When $l=0$, then $\mathbb{C}^{l}$ contains one element, namely the empty sequence, which all $\bar z_\alpha$ then equal to.} Let $O \ni \bar z_\beta$ be a compact neighborhood of $\bar z_\beta$ so that every element of $O$ is one-to-one and does not have $z$ in any coordinate. This is easily possible as $\bar z_\beta$ is one-to-one and $z \notin \{ z_{\beta,i} : i < l \} = a_{p_\beta(0)} \setminus \{z\}$. According to Lemma~\ref{lem:interbound}, and considering $K = O \times \{z\}$, there is $L > 0$ such that for any $\bar z \in O$, there is $g \in \mathcal{H}(\mathbb{C})$ such that $g(z) = 0$, $g(z_i) = 1$, for all $i < l$, and $\| g \|_m < L$. Now let $\varepsilon' < \frac{\varepsilon}{2L}$ and for each $\alpha < \omega_1$, let $\bar p_\alpha'$ be such that $p'_{\alpha}(i) = p_\alpha(i)$, for $i > 0$ and $$ p'_{\alpha}(0) = ( a_{p_\alpha(0)} \setminus \{z\}, f_{p_\alpha(0)}, \varepsilon', m).$$

     Then note that $\bar p'_\alpha \in \mathbb{Q}(H_0) \times \dots \times \mathbb{Q}(H_n)$, for every $\alpha$. Thus there are $\gamma < \delta < \omega_1$ such that $\bar z_{\gamma}, \bar z_{\delta} \in O$ and $\bar q' \leq \bar p'_{\gamma}, \bar p'_{\delta}$, for some $\bar q' \in \mathbb{Q}(H_0) \times \dots \times \mathbb{Q}(H_n)$. Let $g \in \mathcal{H}(\mathbb{C})$ be such that $\| g \|_m < L$, $g(z_{\delta, i})= 1$ for every $i < l$ and $g(z) = 0$. Let $k = f_{q'(0)} - f_{p_{\gamma}(0)}$ and consider $f = f_{p_{\gamma}(0)} + g\cdot k$. Note that \begin{align*}
         \| f - f_{p_{\gamma}(0)} \|_m &= \| g\cdot k \|_m  \leq \| g \|_m \cdot \| k \|_m \\ &<L \cdot \frac{\varepsilon}{2L} = \frac{\varepsilon}{2}.
     \end{align*}
As $\| f_{p_{\gamma}(0)} - f_{p_{\delta}(0)}\| < \frac{\varepsilon}{2}$, we immediately find that $$\| f - f_{p_{\delta}(0)} \|_m < \frac{\varepsilon}{2} + \frac{\varepsilon}{2} = \varepsilon.$$

Also, note that since $f_{q'(0)}(z_{\gamma, i}) = f_{p_{\gamma}(0)}(z_{\gamma, i})$, $k(z_{\gamma, i}) = 0$ for every $i < l$. Then by choice of $g$ and the fact that $\bar q' \leq \bar p'_{\gamma}, \bar p'_{\delta}$, it is easy to compute that $f_{p_{\gamma}(0)} \restriction a_{p_{\gamma}(0)} \subseteq f$ and $f_{p_{\delta}(0)} \restriction a_{p_{\delta}(0)} \subseteq f$. 
Thus, if we let $\bar q$ be such that $q(i) = q'(i)$ for $i > 0$, and $$q(0) = (a_{p_{\gamma}(0)} \cup a_{p_{\delta}(0)}, f, \varepsilon^*, m),$$

for some small enough $\varepsilon^* < \varepsilon$, we have that $\bar q \leq \bar p_\gamma, \bar p_\delta$, contradicting our initial assumption. 
 \end{proof}

 Note that by a simple inductive argument, the previous lemma implies that we can extend simultaneously each $H_i$ in countably many arbitrary points with arbitrary countable sets of values and preserve the ccc: 

 \begin{prop}\label{prop:extccc}
      Let $H_0, \dots, H_n$ be such that $\mathbb{Q}(H_0) \times \dots \times \mathbb{Q}(H_n)$ is ccc. Let $H'_0 \supseteq H_0, \dots, H'_n \supseteq H_n$ be such that for every $i \leq n$, \begin{enumerate}
          \item $\dom H'_i \setminus \dom H_i$ is countable, 
          \item and for any $z \in \dom H'_i \setminus \dom H_i$, $H_i'(z)$ is countable.
      \end{enumerate}
      Then $\mathbb{Q}(H'_0) \times \dots \times \mathbb{Q}(H'_n)$ is ccc. 
 \end{prop}

 \subsection{Tools for the limit step}

\begin{lemma}\label{lem:modelapprox}
Let $M\subseteq V$ be a transitive model of $\ZF^-$ (possibly a proper class). Let $f \in \Hol$, $a \in [\mathbb{C}]^{<\omega} \cap M$ such that $f \restriction a \in M$, $K \in [\mathbb{C} \setminus a]^{<\omega} \cap M$ such that $f \restriction K$ is constant, $\varepsilon > 0$ and $m \in \omega$. Moreover, for any $\xi \in \mathbb{C}$, let $$g_\xi(z) = \xi \sum_{x \in K} \prod_{y \in A \setminus \{x\}} \frac{(z-y)}{(x-y)},$$ where $A = a \cup K$. Then there is $\tilde f \in \Hol \cap M$ and $\delta > 0$ such that $f\restriction a \subseteq \tilde f$, $\tilde f \restriction K$ is constant, $\| \tilde f - f \|_m < \varepsilon$ and 
\begin{enumerate}
    \item $\forall \xi \in B_\delta(0) ( \| \tilde f + g_\xi  - f \|_m < \varepsilon)$,
    \item $ \exists \xi \in B_\delta(0) \forall z \in K (\tilde f(z) + g_\xi(z) = f(z))$.
\end{enumerate}
Whenever $f \in M$, we can assume that $\tilde f = f$. 
\end{lemma}

Before we continue to the proof let us note that $g_\xi \in \Hol$ is simply a function such that $g_\xi \restriction K$ is constantly $\xi$ and $g_\xi(z) = 0$, for $z \in a$. Also, the smaller the absolute value of $\xi$ is, the smaller $\| g_\xi \|_m$ gets.

\begin{proof}
We may assume, without losing generality, that $m$ is large enough such that $A \subseteq B_m(0)$. Let $\delta < \frac{\varepsilon}{2}$ be small enough so that for any $\xi \in B_\delta(0)$, $\| g_\xi \|_m < \frac{\varepsilon}{2}$. Since $f \restriction a \in M$, we can easily find a function $\tilde f \in \Hol \cap M$ such that $f \restriction a \subseteq \tilde f$, $\tilde f$ is constant on $K$ and $\| \tilde f - f \|_m < \delta < \varepsilon$. If $f \in M$ already, we may simply use $\tilde f = f$. For (1), \begin{align*}
    \| \tilde f + g_\xi  - f \|_m &\leq \| \tilde f  - f \|_m + \| g_\xi \|_m \\ &< \delta + \delta < \frac{\varepsilon}{2} + \frac{\varepsilon}{2} = \varepsilon.
\end{align*}

For (2), define $\xi = f(z)- \tilde f(z)$, for any, equivalently every, $z \in K$. Then $\xi \in B_{\delta}(0)$ and for any $z \in K$, $\tilde f(z) + g_\xi(z) = f(z).$
\end{proof}

For a set $F$ of finite partial functions on $\mathbb{C}$, let $$\mathbb{Q}_0(F) := \{ p \in \mathbb{Q} : f_p \restriction a_p \in F \}$$ Forcings $\mathbb{Q}(H)$ are a particular type of forcings of the form $\mathbb{Q}_0(F)$. As with $\mathbb{Q}(H)$, the interpretation of $\mathbb{Q}_0(F)$ in any transitive model that contains $F$ is easily seen to be a dense subset of $\mathbb{Q}_0(F)$ as interpreted in $V$. The following lemma will be formulated for this more general type of forcing, since that will be needed in Section~\ref{sec:properuniv}. The proof of this core lemma originates in ideas fleshed out in Burke's \cite{Burke2009}. For instance, Claim~2.8 in the aforementioned paper corresponds roughly to Claim~\ref{clm:claim} below. Burke remarks that this is a version of an argument by Shelah from \cite{Shelah1980}.

\begin{lemma}\label{lem:cohen}
Let $M \subseteq V$ be as in Lemma~\ref{lem:modelapprox}, $F \in M$, $z \in \mathbb{C}\cap (M \setminus \bigcup_{h \in F} \dom h)$ and $c$ a Cohen real over $M$.\footnote{I.e. $c$ is in any open dense subset of $\mathbb{C}$ coded in $M$.} Furthermore, let $F' = F \cup \{ h \cup \{(z,c)\} : h \in F \}$ and $\mathbb{P} \in M$ be a forcing notion that is dense in a forcing $\mathbb{P}'\in V$. Then $$\mathbb{Q}_0(F) \times \mathbb{P} \lessdot_M \mathbb{Q}_0(F') \times \mathbb{P}'.$$
\end{lemma}

\begin{proof}
It is easy to see that $\mathbb{Q}_0(F) \times \mathbb{P}$ is a subforcing of $\mathbb{Q}_0(F') \times \mathbb{P}'$ (the incompatibility relation is preserved). Now let $E \in M$, $E \subseteq \mathbb{Q}_0(F) \times \mathbb{P}$ be predense in $\mathbb{Q}_0(F) \times \mathbb{P}$ and suppose towards a contradiction there exists $\bar p = (p_0, p_1) \in \mathbb{Q}_0(F') \times \mathbb{P}'$, such that $\bar p \perp E$, where $z \in a_{p_0}$ and so $f_{p_0}(z) = c$. By extending $\bar p$, we may assume without loss of generality that $p_1 \in \mathbb{P}$ and that $a_{p_0} \subseteq B_{m_{p_0}}(0)$. 

Let $a := a_{p_0} \setminus \{ z \}$, $K := \{z\}$, $f := f_{p_0}$, $\varepsilon := \frac{\varepsilon_{p_0}}{4}$, $m := m_{p_0}$ and apply Lemma~\ref{lem:modelapprox} to find $\tilde f \in M$ and $\delta > 0$ as in the conclusion of the lemma. Let $\tilde p_0 := (a, \tilde f, \frac{\varepsilon_{p_0}}{2}, m)$. Then $(\tilde p_0, p_1) \in (\mathbb{Q}_0(F) \times \mathbb{P}) \cap M$. Since $c \in B_{\delta}(\tilde f(z))$, we may find a basic open set $O \subseteq B_{\delta}(\tilde f(z))$ such that $c \in O$. 

In the following, for a condition $p$ and a subset $E$ of a poset, we write $p \leq E$ to mean that $p$ extends some element of $E$.

\begin{claim}\label{clm:claim}
There is a dense open set $U \subseteq O$ coded in $M$ so that for every $d \in U$, there is $\bar q \in\mathbb{Q}_0(F) \times \mathbb{P}$ with $\bar q \leq E,(\tilde p_0,  p_1)$ and $f_{q_0}(z) = d$.
\end{claim}

Once we prove the claim we are done. Namely, as $c \in O$ is Cohen generic over $M$, $c \in U$. Then, according to the claim, there is $\bar q \leq E,(\tilde p_0, p_1)$ such that $f_{q_0}(z) = c$. Letting $r_0 := (a_{r_0}, f_{q_0}, \varepsilon_{q_0}, m_{q_0})$, where $$a_{r_0} := a_{q_0} \cup \{ z\},$$ we clearly have that $\bar r = (r_0, q_1) \leq \bar q$ and $\bar r \in \mathbb{Q}_0(F') \times \mathbb{P}'$. Moreover, we have that $r_0 \leq p_0$ and thus $\bar r \leq \bar p, E$: \begin{align*}
    \| f_{p_0} - f_{r_0} \|_{m_{p_0}} &= \| f_{p_0} - f_{q_{0}} \|_{m} \leq \| f_{p_0} - \tilde f \|_{m} + \| \tilde f - f_{q_0} \|_{m} \\ &< \frac{\varepsilon_{p_0}}{4} + \left(\frac{\varepsilon_{p_0}}{2} - \varepsilon_{q_0}\right) = \frac{3 }{4}\varepsilon_{p_0} - \varepsilon_{r_0} \\ &< \varepsilon_{p_0} - \varepsilon_{r_0}.
\end{align*}

This contradicts the assumption that $\bar p \perp E$.

\begin{proof}[Proof of Claim.]
Work in $M$. Let $O_0 \subseteq O$ be an arbitrary non-empty open set. We will find a non-empty open set $O_1 \subseteq O_0$ that will be included in $U$. Let $e \in O_0$ be an arbitrary rational complex number. Then we find $\xi \in B_{\delta}(0)$ such that $f_{\xi}(z) = e$, where $$f_{\xi} = \tilde f + g_{\xi}$$ and $g_{\xi}$ is the function from the statement of Lemma~\ref{lem:modelapprox}. Then, by (1) of the lemma, \begin{align*}
\varepsilon^* := \| f_{\xi} - \tilde f \|_{m_{p_0}} &\leq \| f_{\xi} - f_{p_0} \|_{m} + \| f_{p_0} - \tilde f \|_{m}  \\ &<\frac{\varepsilon_{p_0}}{4} + \frac{\varepsilon_{p_0}}{4} = \varepsilon_{\tilde p_0}. \end{align*}

Consider a condition $q'_0 = (a, f_{\xi}, \varepsilon_{q'_0}, m) \in \mathbb{Q}_0(F)$ where $\varepsilon_{q'_0} < \varepsilon_{\tilde p_0} - \varepsilon^*$ and $\varepsilon_{q'_0}$ is small enough so that $B_{2\varepsilon_{q'_0}}(e) \subseteq O_0$. Then $q'_0 \leq \tilde p_0$ as $$\| f_{\xi} - \tilde f \|_{m_{\tilde p_0}} = \varepsilon^* = \varepsilon_{\tilde p_0} - ( \varepsilon_{\tilde p_0} - \varepsilon^*) < \varepsilon_{\tilde p_0} - \varepsilon_{q'_0}.$$ 

In particular, $(q'_0, p_1) \leq (\tilde p_0, p_1)$. Now let $\bar q'' \in \mathbb{Q}_0(F) \times \mathbb{P}$ be such that $\bar q'' = (q''_0, q''_{1}) \leq E, (q'_0, p_1)$. Let $\gamma \in (0,\varepsilon_{q'_0})$ be small enough so that for any $v \in B_{\gamma}(0)$ there is an entire function $h_{v}$ with $\|h_{v}\|_{m_{q''_0}} < \varepsilon_{q''_0}$, $h_{v}(z) = v$ and $h_{v} \restriction a_{q''_0}$ constantly equals $0$ (e.g. using a similar formula as in Lemma~\ref{lem:modelapprox}). 

As $\| f_{q''_0} - f_{\xi} \|_{m_{q'_0}} < \varepsilon_{q'_0}$, we have that for any $v \in B_{\gamma}(0)$ and $h_v$ as above, \begin{align*}\| (f_{q''_0} + h_{v}) - f_\xi \|_{m_{q'_0}} & \leq \|f_{q''_0} - f_{\xi} \|_{m_{q'_0}} + \|h_{v} \|_{m_{q'_0}} \\ &< \varepsilon_{q'_0} + \varepsilon_{q''_0} \leq  2\varepsilon_{q'_0}.\end{align*}

Now we let $O_1 := B_\gamma(f_{q''_0}(z)) \subseteq B_{2\varepsilon_{q'_0}}(e) \subseteq O_0$. The inclusion follows, since for any $d \in B_\gamma(f_{q''_0}(z))$, $\vert d - e \vert = \vert f_{q''_0}(z) + h_v(z) - f_\xi(z) \vert < 2\varepsilon_{q'_0}$ for some $v \in B_\gamma(0)$. $U$ is constructed in $M$ as the union of all sets $O_1$ that we obtain in this way.

Let us check that this works. So working in $V$, let $d \in O_1$ be arbitrary. Then there is $v \in B_\gamma(0)$ and $h_v$ as before so that $f_{q''_0}(z) + h_v(z) = d$.

Let $q_0 = (a_{q_0}, f_{q_0}, \varepsilon_{q_0}, m_{q_0})$, where $a_{q_0} = a_{q''_0}$, $f_{q_0} = f_{q''_0} + h_v$, $$\varepsilon_{q_0} < \varepsilon_{q''_0} - \| f_{q_0} - f_{q''_0} \|_{m_{q''_0}}$$ and $m_{q_0} = m_{q''_0}$. Then $f_{q''_0} \restriction a_{q''_0} = f_{q_0} \restriction a_{q_0}$ as $h_{v} \restriction a_{q''_0}$ is constantly $0$. So if we let $\bar q = (q_0, q''_1)$, then $\bar q \in \mathbb{Q}_0(F) \times \mathbb{P}$ and $\bar q \leq \bar q'' \leq E, (\tilde p_0, p_1)$: \begin{align*}
                 \|f_{q_0} - f_{q''_0} \|_{m_{q''_0}} &= \varepsilon_{q''_0} - (\varepsilon_{q''_0} - \| f_{q_0} - f_{q''_0} \|_{m_{q''_0}})  \\ &< \varepsilon_{q''_0} - \varepsilon_{q_0}.                                    \end{align*}

Moreover, $f_{q_0}(z) = d$ as required. This finishes the proof of the claim.\end{proof}\end{proof}

\begin{prop}\label{prop:cohen}
Let $M \subseteq V$ be as in Lemma~\ref{lem:modelapprox}, $H_0, \dots, H_n \in M$ and $H_0' \supseteq H_0, \dots, H_n' \supseteq H_n$ be partial functions from $\mathbb{C}$ to $\mathcal{P}(\mathbb{C})$ such that 

\begin{enumerate}
    \item $\dom(H'_i) \setminus \dom(H_i) \subseteq M$, for all $i \leq n$, 
    \item for any pairwise distinct pairs $(i_j,z_j)$, $j \leq l$, where $i_j \leq n$, $z_j \in \dom(H'_{i_j}) \setminus \dom(H_{i_j})$, and any sequence $\langle c_j : j \leq l \rangle$, $c_j \in H'_{i_j}(z_j)$, we have that $\langle c_j : j \leq l \rangle$ is mutually Cohen generic over $M$.\footnote{I.e., $\langle c_j : j \leq l \rangle$ is in any open dense subset of $\mathbb{C}^{l+1}$ coded in $M$.}
\end{enumerate}

Then $\mathbb{Q}(H_0) \times \dots \times \mathbb{Q}(H_n) \lessdot_M \mathbb{Q}(H'_0) \times \dots \times \mathbb{Q}(H'_n).$
\end{prop}

\begin{proof}
    This is an inductive argument using Lemma~\ref{lem:cohen}. Let $E \in M$ be predense in $\mathbb{Q}(H_0) \times \dots \times \mathbb{Q}(H_n)$. Towards a contradiction, suppose that $l$ is minimal such that there is $\bar p \in \mathbb{Q}(H'_0) \times \dots \times \mathbb{Q}(H'_n)$, $\bar p \perp E$ and $\{(i, z) : i \leq n, z \in a_{p_i} \setminus \dom(H_i)\}$ is enumerated by a sequence $\langle (i_j,z_j) : j \leq l \rangle$. Then according to (2), $\langle c_j : j \leq l \rangle = \langle f_{p_{i_j}}(z_j) : j < l \rangle$ is mutually Cohen generic over $M$. In $M[c_0, \dots, c_{l-1}]$, consider $$H''_i = H_i \cup \{ (z_j,\{ c_j \} ) : j < l, i_j = i \},$$ for every $i \leq n$. Then, by the minimality of $l$, $E$ is still predense in $\mathbb{Q}(H''_0) \times \dots \times \mathbb{Q}(H''_n)$. Now apply Lemma~\ref{lem:cohen} once to accommodate the full condition $\bar p$ and reach a contradiction.
\end{proof}

\subsection{The main theorem}

\begin{thm}\label{thm:main}
(GCH) Let $\kappa$ be an infinite cardinal of uncountable cofinality. Then there is a cofinality and cardinal preserving forcing extension in which 

\begin{enumerate}
    \item $2^{\aleph_0} = \kappa$, 
    \item there is a Wetzel family, 
    \item if $\kappa$ is regular, $\MA$ holds.
\end{enumerate}
\end{thm}

\begin{proof}
    Start with the model obtained in Proposition~\ref{prop:Baumgartner}, where there is an almost disjoint sequence $\langle \sigma_\alpha : \alpha < \kappa \rangle$ in $\prod_{\xi < \kappa} \mu_\xi$, $\mu_\xi = \max(\vert\xi\vert, \aleph_0)$. This is our ground model $V$ now. Fix a bookkeeping function $B$ with domain $\kappa$. The details of $B$ are going to be discussed at the end. We are going to recursively define a ccc finite support iteration $\mathbb{P} = \langle \mathbb{P}_\alpha, \mathbb{Q}_\alpha : \alpha < \kappa \rangle$. Additionally, there will be the following objects for each $\alpha < \kappa$: 
    \begin{enumerate}
        \item $\mathbb{P}_{\alpha+1}$-names $\dot C_{\alpha, \xi}$, $\xi < \mu_\alpha$, for pairwise disjoint countable dense sets of complex numbers, such that any $ \bar c \in (\bigcup_{\xi < \mu_\alpha} C_{\alpha, \xi})^{<\omega}$ is mutually Cohen generic over $V^{\mathbb{P}_\alpha}$, 
        \item a countable set $X_\alpha \subseteq \kappa$,
        \item a $\mathbb{P}_\alpha$-name $\dot z_\alpha$ for a complex number, 
        \item a $\mathbb{P}_{\alpha+1}$-name $\dot f_\alpha$ for an entire function, such that $\Vdash_{\mathbb{P}_{\alpha +1}} \dot f_\alpha(\dot z_\delta) \in \bigcup_{\xi < \mu_\delta} \dot C_{\delta, \xi}$, for all $\delta < \alpha$. 
    \end{enumerate}

    For any $\sigma \in \prod_{\xi < \alpha} \mu_\xi$, we then let $\dot H_{\sigma}$ be a $\mathbb{P}_\alpha$-name for $$\{ (z_\delta, C_{\delta, \sigma(\delta)}) : \delta < \alpha\}.$$

We will inductively prove that for any $n\in \omega$ and $\xi_0 < \dots < \xi_n \in \kappa \setminus \bigcup_{\delta < \alpha} X_\delta$, \begin{equation}
    \Vdash_{\mathbb{P}_\alpha} \mathbb{Q}(\dot{H}_{\sigma_{\xi_0} \restriction \alpha}) \times \dots \times \mathbb{Q}(\dot{H}_{\sigma_{\xi_n} \restriction \alpha}) \text{ is ccc}. \tag{$*$}
\end{equation}

Start with $\mathbb{P}_0 = \{ \mathbbm{1} \}$. Clearly ($*$) above is satisfied when $\alpha = 0$.\footnote{ $\mathbb{Q}(\emptyset)$ has a countable dense subset, e.g. consisting of those conditions $p$, where $f_p$ is a polynomial in rational coefficients.} Suppose we have constructed $\mathbb{P}_\alpha$ and we showed that ($*$) holds. Then we first define a forcing notion $\mathbb{P}_\alpha^{+}$ extending $\mathbb{P}_\alpha$.

Suppose $B(\alpha)$ is a $\mathbb{P}_\alpha$-name $\dot{\mathbb{A}}$ for a ccc poset of size $< \kappa$. In not, we simply let $\mathbb{P}_\alpha^+ = \mathbb{P}_\alpha$. In $V^{\mathbb{P}_\alpha}$, there may be $\xi_0 < \dots < \xi_n \in \kappa \setminus \bigcup_{\delta < \alpha} X_\delta$ such that $$ \mathbb{A} \times \mathbb{Q}({H}_{\sigma_{\xi_0} \restriction \alpha}) \times \dots \times \mathbb{Q}({H}_{\sigma_{\xi_n} \restriction \alpha}) \text{ is not ccc}.$$ 

Note then, that $\mathbb{Q}({H}_{\sigma_{\xi_0} \restriction \alpha}) \times \dots \times \mathbb{Q}({H}_{\sigma_{\xi_n} \restriction \alpha})$ forces that $\mathbb{A}$ is not ccc and this remains the case in any further ccc extension, by Lemma~\ref{lem:ccctwostep}. Also, this forces that for any $\xi'_0 < \dots < \xi'_m \in \kappa \setminus \bigcup_{\delta < \alpha} X_\delta \cup \{\xi_0 < \dots < \xi_n \}$, $$\mathbb{Q}({H}_{\sigma_{\xi'_0} \restriction \alpha}) \times \dots \times \mathbb{Q}({H}_{\sigma_{\xi'_m} \restriction \alpha}) \text{ is still ccc.}$$  
If no such $\xi_0 < \dots < \xi_n$ exist, then $\mathbb{A}$ preserves that all such products are ccc, again by Lemma~\ref{lem:ccctwostep}. We let $\mathbb{P}_\alpha^+$ be $\mathbb{P}_\alpha * \dot{\mathbb{B}}$, where $\dot{\mathbb{B}}$ is a $\mathbb{P}_\alpha$-name for $\mathbb{Q}({H}_{\sigma_{\xi_0} \restriction \alpha}) \times \dots \times \mathbb{Q}({H}_{\sigma_{\xi_n} \restriction \alpha})$ or $\mathbb{A}$, depending on which of the cases occur. In $V$, using the ccc of $\mathbb{P}_\alpha$, we let $X_\alpha^-$ be a countable set that contains any such $\xi_0, \dots, \xi_n$ that we might choose in the former case. Then ($*$) still holds if we replace $\mathbb{P}_\alpha$ by $\mathbb{P}_\alpha^+$ and if $\xi_0 < \dots < \xi_n$ are in $\kappa \setminus (X_\alpha^- \cup \bigcup_{\delta < \alpha} X_\delta)$, by what we have already noted.

Next we let $\dot z_\alpha = B(\gamma)$, where $\gamma$ is least such that $B(\gamma)$ is a $\mathbb{P}_\alpha$-name for a complex number distinct from any $\dot z_\delta$, $\delta < \alpha$. Let $\mathbb{C}_{\mu_{\alpha}}$ be the forcing for adding mutually generic Cohen reals $\langle c_{\alpha, \xi, i} : \xi < \mu_\alpha, i \in \omega \rangle$. Let $\mathbb{P}_\alpha^{++} = \mathbb{P}_\alpha^+ \times \mathbb{C}_{\mu_{\alpha}}$ and $\dot C_{\alpha, \xi}$ be a $\mathbb{P}_\alpha^{++}$-name for $\{ c_{\alpha, \xi, i} : i \in \omega \}$. For any $\xi_0 < \dots < \xi_n \in \kappa \setminus (X_\alpha^- \cup \bigcup_{\delta < \alpha} X_\delta)$, we still have that $$ \Vdash_{\mathbb{P}_\alpha^{++}} \mathbb{Q}(\dot{H}_{\sigma_{\xi_0} \restriction \alpha}) \times \dots \times \mathbb{Q}(\dot{H}_{\sigma_{\xi_n} \restriction \alpha}) \text{ is ccc},$$ since Cohen forcing preserves the ccc of any poset. Let $\eta_\alpha \in \kappa \setminus (X_\alpha^- \cup \bigcup_{\delta < \alpha} X_\delta)$ be arbitrary, let $X_\alpha = X_\alpha^- \cup \{\eta_\alpha \}$, $$\mathbb{P}_{\alpha +1} =\mathbb{P}_{\alpha}^{++} * \mathbb{Q}(\dot{H}_{\sigma_{\eta_\alpha} \restriction \alpha}),$$ and $\dot f_\alpha$ be a name for the entire function added by $\mathbb{Q}({H}_{\sigma_{\eta_\alpha} \restriction \alpha})$, as described in Lemma~\ref{lem:genericfunction}.

Now, for any $\xi_0 < \dots < \xi_n \in \kappa \setminus (\bigcup_{\delta < \alpha+1} X_\delta)$, $$\Vdash_{\mathbb{P}_{\alpha+1}} \mathbb{Q}(\dot{H}_{\sigma_{\xi_0} \restriction \alpha}) \times \dots \times \mathbb{Q}(\dot{H}_{\sigma_{\xi_n} \restriction \alpha}) \text{ is ccc}$$ and by Proposition~\ref{prop:extccc}, $$\Vdash_{\mathbb{P}_{\alpha+1}} \mathbb{Q}(\dot{H}_{\sigma_{\xi_0} \restriction \alpha +1}) \times \dots \times \mathbb{Q}(\dot{H}_{\sigma_{\xi_n} \restriction \alpha +1}) \text{ is ccc}.$$

So ($*$) holds at $\alpha+1$. It remains to show that ($*$) is preserved in limits $\alpha$. So once again, let $\xi_0 < \dots < \xi_n \in \kappa \setminus (\bigcup_{\delta < \alpha} X_\delta)$. Then there is $\beta < \alpha$ so that $\sigma_{\xi_i}(\delta) \neq \sigma_{\xi_j}(\delta)$ and followingly, $C_{\delta, \sigma_{\xi_i}(\delta)} \cap C_{\delta, \sigma_{\xi_j}(\delta)} = \emptyset$, for any $i < j \leq n$ and $\delta \in [\beta, \alpha)$. Thus, by Proposition~\ref{prop:cohen}, $$ \Vdash_{\delta+1} \mathbb{Q}(\dot{H}_{\sigma_{\xi_0} \restriction \delta}) \times \dots \times \mathbb{Q}(\dot{H}_{\sigma_{\xi_n} \restriction \delta}) \lessdot_{V^{\mathbb{P}_\delta}} \mathbb{Q}(\dot{H}_{\sigma_{\xi_0} \restriction \delta +1}) \times \dots \times \mathbb{Q}(\dot{H}_{\sigma_{\xi_n} \restriction \delta+1}),$$ which, according to Lemma~\ref{lem:comsub}, is equivalent to $$\mathbb{P}_\delta * \mathbb{Q}(\dot{H}_{\sigma_{\xi_0} \restriction \delta}) \times \dots \times \mathbb{Q}(\dot{H}_{\sigma_{\xi_n} \restriction \delta}) \lessdot \mathbb{P}_{\delta +1} * \mathbb{Q}(\dot{H}_{\sigma_{\xi_0} \restriction \delta +1}) \times \dots \times \mathbb{Q}(\dot{H}_{\sigma_{\xi_n} \restriction \delta+1}).$$ By induction it is then easy to see that for any limit $\delta \in [\beta, \alpha]$, $\mathbb{P}_\delta * \mathbb{Q}(\dot{H}_{\sigma_{\xi_0} \restriction \delta}) \times \dots \times \mathbb{Q}(\dot{H}_{\sigma_{\xi_n} \restriction \delta})$ is the direct limit of the forcings $\mathbb{P}_{\delta'} * \mathbb{Q}(\dot{H}_{\sigma_{\xi_0} \restriction \delta'}) \times \dots \times \mathbb{Q}(\dot{H}_{\sigma_{\xi_n} \restriction \delta'})$, for $\delta' \in [\beta, \delta)$ and for $\delta < \alpha$, we know already by ($*$) that $$\mathbb{P}_\delta * \mathbb{Q}(\dot{H}_{\sigma_{\xi_0} \restriction \delta}) \times \dots \times \mathbb{Q}(\dot{H}_{\sigma_{\xi_n} \restriction \delta}) \text{ is ccc.\footnotemark}$$
\footnotetext{To be slightly more accurate, the direct limit is a dense subforcing of $\mathbb{P}_\delta * \mathbb{Q}(\dot{H}_{\sigma_{\xi_0} \restriction \delta}) \times \dots \times \mathbb{Q}(\dot{H}_{\sigma_{\xi_n} \restriction \delta})$.} Thus, by Lemma~\ref{lem:directccc}, $\mathbb{P}_\alpha * \mathbb{Q}(\dot{H}_{\sigma_{\xi_0} \restriction \alpha}) \times \dots \times \mathbb{Q}(\dot{H}_{\sigma_{\xi_n} \restriction \alpha})$ is itself ccc. In particular ($*$) now follows, by Lemma~\ref{lem:ccctwostep}.

This finishes the construction. The bookkeeping function $B$ is supposed to enumerate all $\mathbb{P}$-names for complex numbers, and in case $\kappa$ is regular, also all $\mathbb{P}$-names for ccc forcings on ordinals $< \kappa$ unboundedly often. This is a standard argument. When $\kappa$ is regular, it suffices to let $B$ enumerate all elements of $H(\kappa)$ unboundedly often. A standard argument then shows that $\MA$ holds after forcing with $\mathbb{P}$. When $\kappa$ is not regular, let $B$ enumerate all elements of $\mathcal{P}_\omega^\kappa(\kappa)$, where $\mathcal{P}_\omega^{0}(\kappa)= \kappa$, $\mathcal{P}_\omega^{\alpha+1}(\kappa) = \mathcal{P}_\omega^{\alpha}(\kappa) \cup [\mathcal{P}_\omega^{\alpha}(\kappa)]^{\leq \omega}$, and we take unions at limits. Since $2^{\aleph_0} = \kappa$, $\vert \mathcal{P}_\omega^\kappa(\kappa) \vert = \kappa$ as well. It is then standard to see that $\mathbb{P} \subseteq \mathcal{P}_\omega^\kappa(\kappa)$, using the ccc and choosing appropriate names for $\dot{\mathbb{Q}}_\alpha$, $\alpha < \kappa$.

Finally, after forcing with $\mathbb{P}$, we have that $\langle z_\alpha : \alpha < \kappa \rangle$ enumerates the complex numbers and for every $\delta, \alpha < \kappa$, $$ f_\alpha(z_\delta) \in \{ f_\beta(z_\delta) : \beta \leq \delta \} \cup \bigcup_{\xi < \mu_\delta} C_{\delta, \xi},$$ which has size $\leq \vert \delta \vert + \mu_\delta \cdot \aleph_0 = \mu_\delta < \kappa = 2^{\aleph_0}$. Thus $\mathcal{F} = \{f_\alpha : \alpha < \kappa \}$ is a Wetzel family.  
\end{proof}

Let us note that it is not very important that the $\sigma_\alpha$'s had finite pairwise intersections and we could easily get by with assuming only countable intersections. In that case, we would just have to split up the limit stages into countable and uncountable cofinalities. The proof for uncountable cofinality stays the same and for countable cofinality the ccc follows easily from the previous steps.

Also, some interesting modifications can be made to the forcing construction above. For example, instead of taking care of all complex numbers along the iteration, we can leave out some values. For example, we may leave out exactly $0$. The resulting family $\mathcal{F}$ is then a Wetzel family on the modified domain $\mathbb{C} \setminus \{ 0 \}$, while all values $f(0)$, for $f \in \mathcal{F}$, are pairwise distinct. To see this, note that if $z \not\in \dom H$, then the function added by $\mathbb{Q}(H)$ maps to a generic complex number at $z$. 

More generally, for any given infinite $\Omega \subseteq \mathbb{C}$, we can construct a family of entire functions that is Wetzel on $\Omega$, while attaining $2^{\aleph_0}$-many values at any point outside of $\Omega$. 
For $\kappa$ regular, we can force $\diamondsuit_\kappa$ over the model from Proposition~\ref{prop:Baumgartner} without collapsing cardinals or changing cofinalities, by adding a $\kappa$-Cohen real in the standard way. This doesn't affect the result of Proposition~\ref{prop:Baumgartner}. Then, using a standard guessing argument in the iteration of Theorem~\ref{thm:main}, it should not be hard to modify the construction in order to obtain a model where such families exist for every infinite subset $\Omega \subseteq \mathbb{C}$. Whenever $\Omega_\alpha \subseteq \langle z_\beta : \beta < \alpha \rangle$ is guessed at step $\alpha$, we may force with $\mathbb{Q}(H_{\sigma_{\eta_\alpha} \restriction \alpha} \restriction \Omega_\alpha)$ instead. Here note that if $\mathbb{Q}(H)$ is ccc then for any restriction $H' \subseteq H$, $\mathbb{Q}(H')$ is a subforcing (not necessarily complete) of $\mathbb{Q}(H)$ and thus also ccc.

\section{MA and universal sets}\label{sec:mauniv}
 
In this section, we show that under $\MA + \neg \CH$ there is no universal set. Recall that $\MA$ is saying that for any ccc partial order $\mathbb{P}$ and any family $\mathcal{D}$ of less than $2^{\aleph_0}$-many dense subsets of $\mathbb{P}$, there is a filter $G \subseteq \mathbb{P}$ such that $G \cap D \neq \emptyset$, for every $D \in \mathcal{D}$. 

We begin by introducing the ccc poset that we will use. The forcing $\mathbb{S}$ shall consist of conditions of the form $p = (w,s) = (w_p,s_p)$, where $w \in [\mathbb{C}]^{<\omega}$ and $s$ is a finite sequence of open intervals with rational endpoints in $(0,1) \subseteq \mathbb{R}$ such that 

\begin{enumerate}[label=(\alph*)]
    \item $\forall i < j < \lth{s} \forall x \in s(i) \forall y \in s(j)( y < x^{8} )$, 
    \item $\forall z_0, z_1 \in w (\vert z_0 - z_1\vert \in \bigcup_{i < \lth{s} } s(i) \cup \{ 0\})$.
\end{enumerate}

A condition $q$ extends $p$ iff $w_p \subseteq w_q$ and $s_p \subseteq s_q$.

\begin{lemma}\label{lem:distanceccc}
$\mathbb{S}$ is ccc. 
\end{lemma}

\begin{proof}
    Suppose that $p_\alpha \in \mathbb{S}$, for $\alpha < \omega_1$ are pairwise incompatible. Then we may assume without loss of generality that $s_{p_\alpha} = s$ and $\vert w_{p_\alpha} \vert = n$ is the same for all $\alpha < \omega_1$. Let us write $w_{p_\alpha} = \{ z_i^\alpha : i < n\}$, for every $\alpha$ and consider the function $d \colon \mathbb{C}^n \times \mathbb{C}^n \to [0,\infty)^{n \times n}$, where $$d(\langle z_i : i < n \rangle, \langle z'_i : i < n\rangle ) = \langle \vert z_i - z'_j \vert : (i,j) \in n \times n\rangle.$$
    
    $d$ is clearly continuous. Moreover, since $\{ \langle z_i^\alpha : i < n \rangle : \alpha < \omega_1 \} \subseteq \mathbb{C}^{n}$ is uncountable, there is some $\alpha < \omega_1$ such that $\langle z_i^\alpha : i < n \rangle$ is an accumulation point of that set. Thus let $\alpha_k \neq \alpha$, $k \in \omega$, be such that $\langle z_i^{\alpha_k} : i < n \rangle \rightarrow \langle z_i^\alpha : i < n \rangle$ as $k \rightarrow \infty$. Let $c$ be the left endpoint of $s(\lth{s} - 1)$, i.e. $c$ is the the infimum of $\bigcup_{i < \lth{s}} s(i)$. Now note that $$d(\langle z^\alpha_i : i < n \rangle, \langle z^\alpha_i : i < n \rangle) \in \Big( \bigcup_{i < \lth{s}} s(i) \cup [0, c^8) \Big)^{n \times n}.$$ Thus, by continuity, there is $k$ large enough so that $$d(\langle z^{\alpha_k}_i : i < n \rangle, \langle z^\alpha_i : i < n \rangle) \in \Big( \bigcup_{i < \lth{s}} s(i) \cup [0, c^8) \Big)^{n \times n}.$$ In other words, for all $z_0, z_1 \in w_{p_\alpha} \cup w_{p_{\alpha_k}}$, $\vert z_0 - z_1 \vert \in \bigcup_{i < \lth{s}} s(i) \cup [0, c^8)$. Let $0 < b < c^8$, where $b$ is strictly bigger than the maximal distance between points in $w_{p_\alpha} \cup w_{p_{\alpha_k}}$ that lies in $(0,c^8)$. Similarly, let $0 < a < b$, where $a$ is strictly smaller than the minimal such distance. Letting $I$ be the interval $(a, b)$,  $(w_{p_\alpha} \cup w_{p_{\alpha_k}}, s^\frown I )$ is a condition extending both $p_\alpha$ and $p_{\alpha_k}$, while $\alpha \neq \alpha_k$. This is a contradiction to the assumption that $p_\alpha\perp p_{\alpha_k}$. 
\end{proof}

In the following, $Q$ is the set of rational numbers. 

\begin{lemma}\label{lem:densitydistance}
For every $z \in \mathbb{C}$ and every $n \in \omega$, the sets $D_z = \{ q \in \mathbb{S} : z \in w_q + (Q + iQ)\}$
and $E_n = \{ q \in \mathbb{S} : \lth{s_q} \geq n \}$ are dense in $\mathbb{S}$.
\end{lemma}

\begin{proof}
 Let $p \in \mathbb{S}$ be arbitrary. If $w_p = \emptyset$, then clearly $q = (\{z\}, s_p)$ extends $p$ and lies in $D_z$. Otherwise, there is $z_0 \in w_p$ and a small open neighborhood $O$ of $z_0$ so that for any $z_1 \in O$, there is an extension $q$ of $p$ with $z_1 \in w_q$. This is similar to the argument in the proof of Lemma~\ref{lem:distanceccc}. $O$ clearly contains a rational translate of $z$. The case of $E_n$ is obvious.\end{proof}

\begin{lemma}\label{lem:openuncountable}
 Let $O \subseteq (0,\infty)$ be open containing arbitrarily small values. Then there is an uncountable set $X \subseteq \mathbb{R}$ such that $\{ \vert z_0 - z_1 \vert : z_0, z_1 \in X, z_0 \neq z_1 \} \subseteq O$.
\end{lemma}

\begin{proof}
By recursion construct a Cantor scheme, i.e. a map $\varphi$ from $2^{<\omega}$ to non-empty open intervals of $\mathbb{R}$, such that for every $t \subseteq t'$, $\overline{\varphi(t')} \subseteq \varphi(t)$, $\operatorname{diam}(\varphi(t)) \leq \frac{1}{\lth{t} + 1}$ and $\varphi(t ^\frown 0) \cap \varphi(t^\frown 1) = \emptyset$. Start with $\varphi(\emptyset) = (0,1)$ and given $\varphi(t)$, find $\varphi(t^\frown 0)$ and $\varphi(t^\frown 1)$ such that for every $x \in \varphi(t^\frown 0)$, $y \in \varphi(t^\frown 1)$, $\vert x - y \vert \in O$. This is possible since $O$ is open and contains arbitrarily small numbers $> 0$. Clearly, $X = \bigcap_{n \in \omega} \bigcup_{t \in 2^n} \varphi(t)$ works. 
\end{proof}

\begin{lemma}[essentially {\cite[Prop. 9.4]{Abraham1985}}]\label{lem:noentiredistance}
 Let $X, Y \subseteq \mathbb{C}$, $X$ uncountable, and assume that for every $x \in \{\vert z_0 - z_1 \vert : z_0, z_1 \in X, z_0 \neq z_1 \}, y \in \{\vert z_0 - z_1 \vert : z_0, z_1 \in Y, z_0 \neq z_1 \}$, $$\min(x,y) < \max(x,y)^2.$$ Then there is no non-constant entire function $f$ such that $f'' X \subseteq Y$. 
\end{lemma}

\begin{proof}
Suppose there is such $f$. Since $f$ is non-constant we can find an accumulation point $x \in X$ of $X$ such that $f'(x) \neq 0$. Let $x_n \rightarrow x$, where $x_n \in X$ for every $n \in \omega$. Then there is $\langle n_k : k \in \omega \rangle$ such that for all $k$, $$\vert x_{n_k} - x \vert < \vert f(x_{n_k}) - f(x) \vert^2$$ or for all $k$, $$\vert f(x_{n_k}) - f(x) \vert < \vert x_{n_k} - x \vert^2.$$

The former is impossible since then $$\frac{\vert f(x_{n_k}) - f(x) \vert }{\vert x_{n_k} - x \vert} \geq \frac{\vert f(x_{n_k}) - f(x) \vert }{\vert f(x_{n_k}) - f(x) \vert^2} = \frac{1}{\vert f(x_{n_k}) - f(x) \vert}$$ which does not converge, and from the latter we follow that $$\frac{\vert f(x_{n_k}) - f(x) \vert }{\vert x_{n_k} - x \vert} \leq \frac{\vert x_{n_k} - x \vert^2 }{\vert x_{n_k} - x \vert} = \vert x_{n_k} - x \vert$$ which converges to $0$. This contradicts that $f'(x) \neq 0$. 
\end{proof}

\begin{thm}\label{thm:MAuniv}
$\MA + \neg \CH$ implies that there is no universal set.
\end{thm}
 
\begin{proof}
 Let $Y \subseteq \mathbb{C}$, $\vert Y \vert < 2^{\aleph_0}$. Using $\MA$, find a filter $G \subseteq \mathbb{S}$ intersecting all sets $D_z$ for $z \in Y$ and $E_n$ for $n \in \omega$ from Lemma~\ref{lem:densitydistance}. Let $Z = \bigcup_{p \in G} w_p$, $s = \bigcup_{p \in P} s_p$ and $U = \bigcup_{n \in \omega} s(n)$. Then $Y \subseteq Z + (Q + iQ)$ and $\{\vert z_0 - z_1 \vert : z_0, z_1 \in Z, z_0 \neq z_1 \} \subseteq U$. Let $a_n$ be the left-endpoint of $s(n)$, for every $n \in \omega$ and consider $O = \bigcup_{n \in \omega} (a_n^4, a_n^2)$. Then note that for every $x \in U$, $y \in O$, $$\min(x,y) < \max(x,y)^2.$$
 
 Finally apply Lemma~\ref{lem:openuncountable} to find a set $X \subseteq \mathbb{C}$ of size $\aleph_1$ such that $\{\vert z_0 - z_1 \vert : z_0, z_1 \in X, z_0 \neq z_1 \}\subseteq O$. We claim that there is no entire $f$ such that $f''X \subseteq Y$. Otherwise, as $Y \subseteq Z + (Q + iQ)$, there is an uncountable $X' \subseteq X$ and there are rationals $r_0, r_1$, such that $$f'' X \subseteq Z + (r_0 + ir_1).$$ Clearly, $\{\vert z_0 - z_1 \vert : z_0, z_1 \in Z + (r_0 + ir_1), z_0 \neq z_1 \} = \{\vert z_0 - z_1 \vert : z_0, z_1 \in Z, z_0 \neq z_1 \} \subseteq U$. This contradicts Lemma~\ref{lem:noentiredistance}
\end{proof} 

\begin{cor}
    The existence of a Wetzel family does not imply the existence of a universal set.
\end{cor}

\begin{proof}
    Taking $\kappa = \aleph_2$, this follows from Theorem~\ref{thm:MAuniv} and Theorem~\ref{thm:main}. Note that if we are allowed to assume the existence of a weakly inaccessible cardinal, this already follows from the main result in combination with Proposition~\ref{prop:univsuccessor}.
\end{proof}

\section{\texorpdfstring{A universal set with $2^{\aleph_0} = \aleph_2$}{A universal set with continuum aleph\_2 }}\label{sec:properuniv}

The construction in the following section will use a countable support iteration of proper forcing notions. For more information we refer the reader to \cite[Chapter 31]{Jech1997}. 

\begin{thm}(CH)\label{thm:properuniv}
    There is a proper forcing extension of $V$ preserving all cardinals and cofinalities in which $\mathbb{C}^V$ is universal and $2^{\aleph_0} = \aleph_2$. In particular, the existence of a universal set is consistent with $2^{\aleph_0} = \aleph_2$.
\end{thm}

We will construct a forcing notion that uses models as side-conditions. For now let us fix an arbitrary set $Y$ of complex numbers. We then say that a pair $(M,N)$ is a node, if $(M, \in, Y \cap M),(N, \in, Y \cap N)$ are countable elementary submodels of $(H(\omega_1), \in, Y)$ and $(M,\in,Y\cap M) \in N$. In particular, $M$ and $N$ are transitive. A side condition is a finite set $s = \{ (M_i, N_i) : i \leq n \} $ of nodes, where $(M_i, N_i) \in M_{i+1}$ for every $i < n$.

Recall the poset $\mathbb{Q}$ from Definition~\ref{def:Q}. The forcing $\mathbb{P}(Y)$ then consists of pairs $(w,s)$ where $s = \{ (M_i, N_i) : i \leq n \}$ is a side condition, $w = (a,f,\varepsilon, m) \in M_n \cap \mathbb{Q}$, $a \subseteq M_{n-1}$ and for every $i < n$, $\langle f(z) : z \in a \cap (M_i \setminus \bigcup_{j < i} M_j) \rangle \subseteq M_{i+1} \cap Y$ is mutually Cohen generic over $N_i$. In other words, the elements of $a$ that appear in a model $M_i$, but not before, are mapped to mutually generic-over-$N_i$ complex numbers that lie in $Y \cap M_{i+1}$. A condition $(v,t)$ extends $(w,s)$ if $v$ extends $w$ in $\mathbb{Q}$ and $s \subseteq t$. 

\begin{lemma}\label{lem:proper}
    $\mathbb{P}(Y)$ is proper. 
\end{lemma}

\begin{proof}
Let $K \preccurlyeq H(\theta)$ be countable with  $Y \in K$, for some large $\theta$. Furthermore, let $K \in K^+ \preccurlyeq H(\theta)$ be another countable model. Consider $M = K \cap H(\omega_1)$ and $N = K^+ \cap H(\omega_1)$. Then $(M,N)$ is a node, since both $M$ and $Y \cap M$ are hereditarily countable in $K^+$ and thus elements of $N$. Elementarity of $(M, \in, Y \cap M)$ and $(N,\in, Y \cap N)$ follows easily from the definability of $H(\omega_1)$ within $K, K^+$ and the elementarity of $K, K^+$. Any subset of $M$ that lies in $K^+$ is an element of $N$, since $M$ is countable in $K^+$. In particular, $\mathbb{P}(Y) \cap K = \mathbb{P}(Y) \cap M \in N$. Moreover, for any subset $A \in K$ of $\mathbb{P}(Y)$, $A \cap K = A \cap M \in N$. 

In the following, $r \parallel A$ means that $r$ is compatible with some element of $A$.
\begin{claim}\label{claim:proper}
    Let $A \in N$ be any predense subset of $\mathbb{P}(Y) \cap M$ and $r$ be any condition of the form $r = (w, s \cup \{ (M,N) \} \cup t)$, where $s \in M$ and $t \cap M = \emptyset$. In other words, $r$ is any condition that contains $(M,N)$ in its side condition. Then $r \parallel A$.
\end{claim}
    \begin{proof}
    We proceed by induction on the length of $t$. If $t$ is empty, then $(w, s)\in \mathbb{P}(Y) \cap M$. By assumption $(w,s)   \parallel A$ since $A$ is predense in $\mathbb{P}(Y) \cap M$. Thus there is $(w',s') \leq (w,s)$ in $M$, extending an element of $A$. Then $(w',s' \cup \{(M,N\}))$ is a condition extending $(w,s \cup \{(M,N) \})$ and that same element of $A$.
    
    Now suppose that $t = t_0 \cup \{ (M_1, N_1)\}$, where $t_0 \in M_1$ and let $(M_0, N_0)$ be the last node of $t_0$, or $(M,N)$ if $t_0 =\emptyset$. Let $$\mathbb{Q}_0 = \begin{cases}
        \{ w' \in \mathbb{Q} : \exists s' \in M, s \subseteq s' ( (w',s') \in \mathbb{P}(Y)) \}&\text{ if } t_0 = \emptyset\\
        \{ w' \in \mathbb{Q} : \exists s' \in M, s \subseteq s' (w',s' \cup \{(M,N) \} \cup t_0) \in \mathbb{P}(Y)\} &\text{ otherwise. }
    \end{cases}$$ Then note that $\mathbb{Q}_0 \in N_0$ is a forcing of the form $\mathbb{Q}_0(F)$ for $F \in N_0$.\footnote{To see that $\mathbb{Q}_0, F \in N_0$, note that the subset of $M_0$ consisting of countable elementary submodels of $(H(\omega_1), \in, Y)$ can be defined in $N_0$ as the elements of $M_0$ that are elementary submodels of $(M_0, \in, Y \cap M_0)$.} Furthermore, let $E = \{ w' \in \mathbb{Q} : \exists s' ((w', s') \in A) \}$. Then $E \subseteq \mathbb{Q}_0$, $E \in N_0$ and by the inductive hypothesis, it is predense in $\mathbb{Q}_0$. Now use Lemma~\ref{lem:cohen}, for $\mathbb{P}, \mathbb{P}'$ trivial forcings, and a similar argument as in Proposition~\ref{prop:cohen} to finish the inductive step.
    \end{proof} Finally, if $(w,s) \in K$ is an arbitrary condition, then we have shown that $(w, s \cup \{(M,N)\})$ is a master condition over $K$. This proves the lemma.\end{proof}

\begin{lemma}\label{lem:cohenpres}
   Let $p, Y \in K \preccurlyeq H(\theta)$, $K$ countable, for large $\theta$, and let $c$ be a Cohen real over $K$. Then there is a master condition $q \leq p$ over $K$ so that $q \Vdash c \text{ is Cohen over } K[\dot G]$.
    \end{lemma}

    This is known as ``almost preserving $\sqsubseteq^{\operatorname{Cohen}}$" in \cite[Section 6.3.C]{Bartoszynski} and implies the preservation of non-meager sets in countable support iterations. 
    
\begin{proof}
    Let $M = K \cap H(\omega_1)$ and $H$ be a $\operatorname{Coll}(\omega, \lambda)$-generic over $K[c]$, where $\lambda = \vert H(\omega_1) \vert^K$. Then $c$ is still a Cohen real over $K[H]$ and $K[H] \models \vert M \vert = \omega$. Moreover, $\mathbb{P}(Y) \cap K = \mathbb{P}(Y) \cap M \in K[H]$, since $\mathbb{P}(Y) \cap M$ is definable from $M, Y \cap M \in K[H]$. Now let $K^+ \preccurlyeq H(\theta)$ be countable with $K, H, c \in K^+$. Let $p = (w,s)$ and $N = K^+ \cap H(\omega_1)$. Then $q = (w,s \cup \{ (M,N) \})$ is a master condition over $K$, as in the proof of Lemma~\ref{lem:proper}. Let $G \ni q$ be $\mathbb{P}(Y)$-generic over $V$. According to Claim~\ref{claim:proper}, for any predense subset $A \in K^+$ of $\mathbb{P}(Y) \cap K$, $A \cap G \neq \emptyset$. Thus $G \cap K$ is $\mathbb{P}(Y) \cap K$-generic over $K^+$ and in particular over $K[H][c] \subseteq K^+$. But $\mathbb{P}(Y) \cap K$ is a countable forcing in $K[H]$ and thus equivalent to Cohen forcing (or, in the simplest case, a trivial forcing) witnessed through an isomorphism in $K[H]$. So $K[H][c][G \cap K]$ is a Cohen extension of $K[H][c]$, and $c$ is still Cohen generic over $K[H][G \cap K]$. In particular $c$ is still Cohen generic over $K[G]  = K[G \cap K] \subseteq K[H][G \cap K]$. 
   \end{proof}
    
\begin{lemma}
    Let $Y \subseteq \mathbb{C}$ be everywhere non-meager. Then $\mathbb{P}(Y)$ generically adds a non-constant entire function $f$ such that $f(\mathbb{C}^V) \subseteq Y$.
\end{lemma}

\begin{proof}
    Let $(w,s) \in \mathbb{P}(Y)$ and $z \in \mathbb{C}$. Extending $(w,s)$ further, we can assume that $z \in M_0$ for some $(M_0,N_0) \in s$, where $M_0$ is minimal with this property and there is a successor $(M_1, N_1) \in s$ of $(M_0,N_0)$. Since $Y$ is everywhere non-meager, $M_1$ knows this and since $N_0$ is countable in $M_1$, there is $c \in Y \cap M_1$ that is arbitrarily close to $f_w(z)$ and Cohen generic over $N_0[c_0, \dots, c_n]$, where $c_0, \dots, c_n$ enumerates the mutual Cohen generics $f_w(x)$, for $x \in a_w$ that first appear in $M_0$. The rest then follows as in the proof of Lemma~\ref{lem:genericfunction}.
\end{proof}

\begin{proof}[Proof of Theorem~\ref{thm:properuniv}]
    Let $Y = \mathbb{C}^V$ and iterate $\mathbb{P}(Y)$ in a countable support iteration of length $\omega_2$. By Lemma~\ref{lem:cohenpres} and \cite[Lemma 6.3.17, 6.3.20]{Bartoszynski}, $Y$ stays everywhere non-meager along the iteration. Everything else follows from standard counting of names arguments and the fact that $\mathbb{P}(Y)$ has size $\aleph_1$ under $\CH$.
\end{proof}

It seems that it would also suffice to consider nodes where $M,N$ are merely countable transitive models of $\ZF^-$, $(M, \in, Y \cap M)$ is countable in $N$ and $Y \cap M$ is non-meager in $M$. This has the slight advantage that $\mathbb{P}(Y) \cap M$ is already a definable subclass of $M$ and not just definable in $N$. Also in the proof of Lemma~\ref{lem:cohenpres} we could directly let $K^+ = K[H][c]$.

There is a also a ccc way to do this this starting from a model of $\diamondsuit$. This is a modification of the construction presented in \cite{Burke2009}, from which our result draws its main inspiration. Instead of generically adding an $\in$-increasing sequence of nodes using side conditions, we start already with a given sequence $\langle (M_\alpha, N_\alpha) : \alpha < \omega_1 \rangle$, where $\langle N_\alpha : \alpha < \omega_1 \rangle$ is an ``oracle" (see \cite{Shelah1998}). It can then be shown that the resulting forcing, consisting of those $w \in \mathbb{Q}$ such that $(w, s) \in \mathbb{P}(Y)$ for every $s \subseteq \langle (M_\alpha, N_\alpha) : \alpha < \omega_1 \rangle$, is ``oracle-cc". In fact, our proper forcing is built directly from this construction. The advantage is that it has a much easier setup and does not depend on a particular chosen sequence of nodes. Also there might be a bigger potential of generalising it to continuum higher than $\aleph_2$, although this is not very clear to us.

\section{Open questions}

\begin{quest}\label{quest:PFAMA}
    Does $\MA$ or $\PFA$ imply that there is a Wetzel family? Is $\MA + 2^{\aleph_0} = \aleph_2$ sufficient?
\end{quest}

Recall that $\operatorname{non}(\mathcal{M})$ is the least size of a non-meager set.

\begin{quest}
   Is every universal set non-meager under $\neg \CH$? In particular, can we replace $\MA$ with $\operatorname{non}(\mathcal{M}) = 2^{\aleph_0}$ in Theorem~\ref{thm:MAuniv}?   
\end{quest}

\begin{quest}
    Is the existence of a universal set consistent with $2^{\aleph_0} = \aleph_3$? With $2^{\aleph_0} = \kappa$ for arbitrary successor cardinal $\kappa$?
\end{quest}

Recall that a domain $\Omega \subseteq \mathbb{C}$ is any open connected subset of $\mathbb{C}$. We may then define the analoguous notion of Wetzel families on $\Omega$ for functions that are holomorphic on $\Omega$. 

\begin{quest}
Let $\Omega \subset \mathbb{C}$ be any domain and suppose that there is a Wetzel family on $\Omega$. Does there exist a Wetzel family on the whole of $\mathbb{C}$? What about $\Omega = \mathbb{C} \setminus \{0\}$?
\end{quest}

\begin{quest}
Can we characterize when $\mathbb{Q}(H)$ is ccc? 
\end{quest}

\bibliographystyle{plain}

\end{document}